\documentclass{siamltex}%
\usepackage{amsfonts}
\usepackage{amsmath}
\usepackage{amssymb}
\usepackage{graphicx}
\usepackage{enumerate}
\usepackage{color}
\usepackage{hyperref}%
\newtheorem{remark}[theorem]{Remark}

\begin{document}

\title{Control semiflows, chain controllability, and the Selgrade decomposition for
linear delay systems}
\author{Fritz Colonius\\Institut f\"{u}r Mathematik, Universit\"{a}t Augsburg, Augsburg, Germany\\fritz.colonius@uni-a.de\\orcid.org/0000-0003-3853-6065}
\maketitle

\textbf{Abstract: }A continuous semiflow is introduced for linear control
systems with delays in the states and controls and bounded control range. The
state includes the control functions. It is proved that there exists a unique
chain control set which corresponds to the chain recurrent set of the
semiflow. The semiflow can be lifted to a linear semiflow on an infinite
dimensional vector bundle with chain transitive base flow. A decomposition
into exponentially separated subbundles is provided by a recent generalization
of Selgrade's theorem.



\section{Introduction\label{Section1}}

We will associate a control semiflow to linear systems with delays in the
states and controls and study their generalized controllability properties.
The considered systems are controlled (retarded) differential delay equations
of the form%
\begin{align}
\dot{x}(t)  &  =A_{0}x(t)+\sum_{i=1}^{p}A_{i}x(t-h_{i})+B_{0}u(t)+\sum
_{i=1}^{p}B_{i}u(t-h_{i}),\,u\in\mathcal{U},\label{D1}\\
x(0)  &  =r,\,x(s)=f(s)\text{ for almost all }-h\leq s\leq0\text{ and
}u(t)=0\text{ for }t\leq0.\nonumber
\end{align}
Here, $A_{0},\ldots,A_{p}\in\mathbb{R}^{n\times n},B_{0},\ldots,B_{p}%
\in\mathbb{R}^{n\times m},0=:h_{0}<h_{1}<\cdots<h_{p}=:h,r\in\mathbb{R}%
^{n},\,f\in L^{2}([-h,0],\mathbb{R}^{n})$, and the set $\mathcal{U}$ of
admissible control functions is given by%
\[
\mathcal{U}:=\{u\in L^{\infty}(\mathbb{R},\mathbb{R}^{m})\left\vert
u(t)\in\Omega\text{ for almost all }t\in\mathbb{R}\right.  \},
\]
for a nonvoid compact and convex set $\Omega\subset\mathbb{R}^{m}$. The unique
solutions $x(t)=\psi(t,r,f,u)$ are absolutely continuous on every interval
$[0,T],T>0$.

A classical topic in control theory is approximate controllability for these
systems, where the states at time $t\geq0$ are $(x(t),x_{t})$ with
$x_{t}(s):=x(t+s),s\in\lbrack-h,0]$ in the state space,%
\[
M_{2}:=M_{2}([-h,0],\mathbb{R}^{n})=\mathbb{R}^{n}\times L^{2}%
([-h,0],\mathbb{R}^{n});
\]
cf. Manitius \cite{Man81}, Curtain and Zwart \cite{CurtZ}, Bensoussan, Da
Prato, Delfour, and Mitter \cite[Chapter 4]{BenDDM}. The problem is to
determine when for given initial state $(x(0),x_{0})=(r,f)\in M_{2}$ it
follows that the reachable set,%
\[
\{(x(t),x_{t})=(\psi(t,r,f,u),\psi_{t}(\cdot,r,f,u))\in M_{2}\left\vert
\text{\ }t\geq0\text{ and }u\in\mathcal{U}\right.  \}
\]
is dense in $M_{2}$. Recently it has found renewed interest by the
contribution of Hinrichsen and Oeljeklaus \cite{HinOel22}. They show (for
systems without control constraints and without delays in the controls) that
robustness of approximate controllability with respect to perturbations
requires the assumption that $\mathrm{rank}(B_{0},A_{p}B_{0},\ldots
A_{p}^{n-1}B_{0})=n$. We will analyze subsets of the state space where chain
controllability holds, which is a weaker version of approximate
controllability in infinite time (cf. Definition \ref{Definition_chain}) and
sometimes may be difficult to distinguish from it in numerical computations.
It allows for small jumps in the trajectories and hence it is not a physical
notion. In the theory of dynamical systems analogous constructions have been
quite successful in order to describe the limit behavior as time tends to
infinity for complicated flows; cf., e.g., Robinson \cite{Robin98}, Alongi and
Nelson \cite{AlonN07}. For finite dimensional control systems the monographs
Colonius and Kliemann \cite{ColK00} and Kawan \cite{Kawan13} contain basic
results on chain controllability; cf. also Da Silva and Kawan \cite{DSilK15},
Ayala, Da Silva, and San Martin \cite{AyaDSSM17} and Da Silva \cite{DaSilva23}.

For control system (\ref{D1}), we construct a continuous affine semiflow
$\Phi$ on the infinite dimensional vector bundle $\mathcal{U}\times M_{2}$ in
the form%
\[
\Phi_{t}(u,r,f)=((u(t+\cdot),(x(t),x_{t}))\text{ for }t\geq0,u\in
\mathcal{U},(r,f)\in M_{2}.
\]
Here, $u(t+\cdot)(s):=u(t+s),s\in\mathbb{R}$ is the right shift and
$\mathcal{U}$ is endowed with a metric compatible with the weak$^{\ast}$
topology of $L^{\infty}(\mathbb{R},\mathbb{R}^{m})$. This generalizes control
flows for finite dimensional systems, cf. \cite{ColK00} and \cite{Kawan13}.

General background on skew product flows (with an emphasis on finite
dimensional systems) is provided by Kloeden and Rasmussen \cite{KloedeR}. For
infinite dimensional systems see Hale \cite{Hale}, Sell and You \cite{SellY02}%
. Some results on chain recurrence for infinite dimensional linear dynamics in
discrete time are presented in Antunez, Mantovani, and Var\~{a}o
\cite{AntMV22}.

The main results of the present paper are the following. Under an injectivity
assumption, Theorem \ref{Theorem_continuous} shows that $\Phi$ is a continuous
semiflow on $\mathcal{U}\times M_{2}$. By Theorem \ref{Theorem_ccs1} there
exists a unique maximal subset of chain controllability, i.e., a chain control
set $E$, in $M_{2}$. The chain control set corresponds to the unique maximal
chain transitive subset of the semiflow $\Phi$; cf. Theorem
\ref{Theorem_equivalence}. The affine semiflow $\Phi$ can be lifted to a
linear semiflow $\Phi^{1}$ on the extended space $\mathcal{U}\times
M_{2}\times\mathbb{R}$. By a theorem due to Blumenthal and Latushkin
\cite{BluL19} the semiflow $\Phi^{1}$ admits a Selgrade decomposition into
exponentially separated subbundles and, equivalently, a Morse decomposition of
the induced flow on the projective bundle $\mathcal{U}\times\mathbb{P}%
(M_{2}\times\mathbb{R})$. This construction is related to the Poincar\'{e}
sphere in the theory of nonlinear differential equations; cf. Remark
\ref{Remark_sphere}. Finally, the special situation is analyzed, where the
linear part of the control system is uniformly hyperbolic. This partially
generalizes pertinent results of Colonius and Santana \cite{ColS24} and
Colonius, Santana, and Viscovini \cite{ColSanV23} in finite dimensions. Kawan
\cite{Kawan17} presents a short review of uniformly hyperbolic finite
dimensional control systems.

Concerning the construction of the control semiflow it is worth to mention
that, for finite dimensional control systems, Desheng Li \cite{Li07} developed
an alternative approach based on differential inclusions, hence avoiding the
explicit use of the space $\mathcal{U}$ of control functions. Here continuity
of the trajectories with respect to the topology on $\mathcal{U}$ plays no role.

The contents of this paper are as follows: Section \ref{Section2} introduces
notation for linear semiflows on infinite dimensional vector bundles and cites
a result by Blumenthal and Latushkin on generalized Selgrade decompositions
for these systems. In Section \ref{Section3} properties of delay equations and
their state space description in $M_{2}$ are recalled. Section \ref{Section4}
constructs the control semiflow for injective delay control systems. Section
\ref{Section5} characterizes chain controllable sets by their lifts to chain
transitive subsets of the control semiflow. Furthermore, it is shown that
there always exists a unique chain control set in $M_{2}$. Section
\ref{Section6} extends the affine delay control system to a linear delay
control system on the state space $M_{2}\times\mathbb{R}$ and applies the
generalized Selgrade theorem from Section \ref{Section2}. Furthermore,
conjugation properties to subsets of the projective bundle are shown. Finally,
Section \ref{Section7} considers the special case of uniformly hyperbolic
systems. Here the affine delay control system is conjugate to its linear part,
and, for the chain control sets in $M_{2}$, stronger results can be
obtained.\medskip

\textbf{Notation}: A semiflow on a metric space $X$ with metric $d$ is a
continuous map $\psi:[0,\infty)\times X\rightarrow X$ with $\psi(0,x)=x$ and
$\psi(t+s,x)=\psi(t,\psi(s,x))$ for $t,s\in\lbrack0,\infty)$ and $x\in X$. A
subset $X^{\prime}\subset X$ is forward invariant, if $\psi(t,x)\in X^{\prime
}$ for all $t\geq0$ and $x\in X^{\prime}$. For a Banach space $Y$ the space of
bounded linear operators on $Y$ is denoted by $\mathcal{L}(Y)$.

\section{Semiflows on Banach bundles\label{Section2}}

This section presents important properties of linear semiflows on Banach
bundles and formulates an infinite dimensional version of Selgrade's theorem
due to Blumenthal and Latushkin \cite{BluL19}.

Let $B$ be a compact metric space with metric $d_{B}$ and let $Y$ be a real
Banach space with norm $\left\Vert \cdot\right\Vert $. Let $\theta
:\mathbb{R}\times B\rightarrow B$ be a continuous flow on $B$, i.e.,
$\theta(0,b)=b,\theta(t+s,b)=\theta(t,\theta(s,b)\,$for $t,s\in\mathbb{R}$ and
$b\in B$. A Banach bundle is given by $\mathcal{V}:=B\times Y$. Consider a
semiflow of injective linear operators over $(B,\theta)$ of the form,%
\[
\Phi:[0,\infty)\times B\times Y\rightarrow B\times Y,\Phi(t,b,y)=(\theta
(t,b),\phi(t,b,s)),t\geq0,b\in B,y\in Y,
\]
such that the following hypotheses hold:

(H1) the projection $\pi_{B}:\mathcal{V}\rightarrow B$ satisfies $\pi_{B}%
\circ\Phi=\theta$.

(H2) For any $(t,b)\in\lbrack0,\infty)\times B$, the map $\Phi(t,b,\cdot
):y\mapsto\Phi(t,b,y)$ is a bounded, injective linear operator of the fibers
$\{b\}\times Y\rightarrow\{\theta(t,b)\}\times Y$.

(H3) For each fixed $t\geq0$, the map $b\mapsto\Phi(t,b,\cdot)$ is continuous
in the operator norm topology on the space $\mathcal{L}(Y)$ of bounded linear
operators on $Y$.

(H4) The mapping $[0,\infty)\times B:(t,b)\mapsto\Phi(t,b,\cdot)$, is
continuous in the strong operator topology on $\mathcal{L}(Y)$, i.e., for all
$y\in Y$ one has $\Phi(t_{k},b_{k},y)\rightarrow\Phi(t_{0},b_{0},y)$ if
$(t_{k},b_{k})\rightarrow(t_{0},b_{0})$.

Where convenient, we will identify the fiber $\{b\}\times Y$ with $Y$ and
write $\Phi_{t}(b,v)=\Phi(t,b,v)$. Note the following consequence of these hypotheses.

\begin{proposition}
\label{Proposition_cont1}Hypotheses (H1)-(H4) imply that $\Phi:[0,\infty
)\times\mathcal{V}\rightarrow\mathcal{V}$ is a continuous mapping in the
metric $d_{\mathcal{V}}$ on $\mathcal{V}$ given by
\begin{equation}
d_{\mathcal{V}}((b_{1},v_{1}),(b_{2},v_{2})):=\max(d_{B}(b_{1},b_{2}%
),\left\Vert v_{1}-v_{2}\right\Vert ). \label{metric_V}%
\end{equation}

\end{proposition}

\begin{proof}
By (H4) it follows that for every $y\in Y$ and $t$ in a compact interval
$I\subset\lbrack0,\infty)$ the set $\left\{  \left\Vert \Phi_{t}%
(b,y)\right\Vert ,t\in I,b\in B\right\}  $ is bounded. Thus the uniform
boundedness principle implies that $\left\{  \left\Vert \Phi_{t}%
(b,\cdot)\right\Vert ,t\in I,b\in B\right\}  $ is bounded.

Let $(t_{k},b_{k},y_{k})\rightarrow(t_{0},b_{0},v_{0})$ in $[0,\infty)\times
B\times Y$ and consider with the metric (\ref{metric_V}) $d_{\mathcal{V}}%
(\Phi(t_{k},b_{k},y_{k}),\Phi(t_{0},b_{0},y_{0}))$. Then $d_{B}(\theta
(t_{k},b_{k}),\theta(t_{0},b_{0}))\rightarrow0$ by continuity of $\theta$. For
the second component one estimates
\begin{align*}
&  \left\Vert \phi(t_{k},b_{k},y_{k})-\phi(t_{0},b_{0},y_{0})\right\Vert \\
&  \leq\left\Vert \phi(t_{k},b_{k},y_{k})-\phi(t_{k},b_{k},y_{0})\right\Vert
+\left\Vert \phi(t_{k},b_{k},y_{0})-\phi(t_{0},b_{0},y_{0})\right\Vert \\
&  \leq\left\Vert \Phi(t_{k},b_{k},\cdot)\right\Vert \,\left\vert \left\vert
y_{k}-y_{0}\right\vert \right\vert +\left\Vert \phi(t_{k},b_{k},y_{0}%
)-\phi(t_{0},b_{0},y_{0})\right\Vert .
\end{align*}
Since the factors $\left\Vert \Phi_{t_{k}}(b_{k},\cdot)\right\Vert $ remain
bounded the first summand converges to $0$. The second summand converges to
$0$ by (H4).
\end{proof}

We write $\mathbb{P}\mathcal{V}$ for the projective bundle $B\times
\mathbb{P}Y$. Here $\mathbb{P}Y$ is the projective space of $Y$ defined by
$\mathbb{P}Y:=(Y\setminus\{0\})/\sim$, where $v\sim w$ for $v,w\in
Y\setminus\{0\}$ if $v=\lambda w$ for some $\lambda\in\mathbb{R}%
\setminus\{0\}$. The metric on $\mathbb{P}\mathcal{V}$ is defined by
\begin{align}
d_{\mathbb{P}\mathcal{V}}((b_{1},v_{1}),(b_{2},v_{2}))  &  :=\max\{d_{B}%
(b_{1},b_{2}),d_{\mathbb{P}}(v_{1},v_{2})\}\text{ with}\label{metric_P}\\
d_{\mathbb{P}}(v,w)  &  :=\min\left\{  \frac{v}{\left\Vert v\right\Vert
}-\frac{w}{\left\Vert w\right\Vert },\frac{v}{\left\Vert v\right\Vert }%
+\frac{w}{\left\Vert w\right\Vert }\right\}  .\nonumber
\end{align}
Since the operators $\Phi_{t}(b,\cdot)$ are injective by (H2) the linear
semiflow descends to the projectivized semiflow $\mathbb{P}\Phi:[0,\infty
)\times\mathbb{P}\mathcal{V}\rightarrow\mathbb{P}\mathcal{V}$ which is continuous.

Recall the following definition from Blumenthal and Latushkin
\cite[Definitions 2.3 and 2.7]{BluL19}.

\begin{definition}
An asymptotically compact attractor of a semiflow $\psi$ on a metric space $X$
is a compact forward invariant set $A\subset X$ such that for some
$\varepsilon>0$ the following properties hold:

(i) For some $S>0$ we have that $\overline{\psi([S,\infty)\times
B_{\varepsilon}(A))}\subset B_{\varepsilon}(A)$ and%
\[
A=\omega(B_{\varepsilon}(A)):=\left\{  y\in X\left\vert \exists t_{k}%
\rightarrow\infty,\exists x_{k}\in B_{\varepsilon}(A):\psi(t_{k}%
,x_{k})\rightarrow y\right.  \right\}  ;
\]

(ii) for any sequence $t_{k}\rightarrow\infty$ and any sequence of points
$x_{k}\in B_{\varepsilon}(A)$ it follows that $\psi(t_{k},x_{k}),k\in
\mathbb{N}$, has a convergent subsequence.
\end{definition}

The points which have pre-images will be relevant.

\begin{definition}
\label{Definition_entire}For a semiflow $\psi$ on $X$ such that $\psi
(t,\cdot)$ is injective for all $t\geq0$, a point $x\in X$ defines an entire
solution,\ if for all $t>0$ there is $y\in X$ with $\psi(t,y)=x$.
\end{definition}

We slightly abuse notation and write $\psi(-t,x)\in X$ for the pre-image
$\psi(t,\cdot)^{-1}(x)$; by injectivity, $\psi(-t,x)$ is a unique element of
$X$ when the pre-image exists. When we write $\psi(-t,x)$ we tacitly suppose
that this pre-image exists. A set $Y\subset X$ is invariant if $\psi(t,x)\in
Y$ for all $t\in\mathbb{R}$.

For $\varepsilon,\tau>0$ an $(\varepsilon,\tau)$-chain from $x$ to $y$ is
given by $q\in\mathbb{N},\ x_{0}\allowbreak=x,x_{1},\ldots,x_{q}=y$ in $X$,
and $\tau_{0},\ldots,\tau_{q-1}\geq\tau$ with
\[
d(\psi(\tau_{j},x_{j}),x_{j+1})<\varepsilon\text{ }\,\text{for\thinspace
}j=0,\ldots,q-1.
\]

\begin{definition}
\label{Definition_chain_trans}(i) A point $x\in X$ is chain recurrent for
$\psi$, if for all $\varepsilon,\tau>0$ there are $(\varepsilon,\tau)$-chains
from $x$ to $y$. The chain recurrent set $\mathcal{R}$ is the set of all chain
recurrent points. If $\psi(t,\cdot)$ is injective for all $t\geq0$, the entire
chain recurrent set $\mathcal{R}^{\#}$ is the set of all chain recurrent
points $x\in X$ which define entire solutions in $\mathcal{R}$.

(ii) A nonvoid set $Y\subset X$ is chain transitive if for all $x,y\in Y$ and
all $\varepsilon,\tau>0$ there are $(\varepsilon,\tau)$-chains from $x$ to $y$.
\end{definition}

\begin{remark}
For a semiflow with compact state space, it follows that through every point
in the chain recurrent set $\mathcal{R}$ there exists an entire solution in
$\mathcal{R}$, hence $\mathcal{R}^{\#}=\mathcal{R}$. This follows as in
Bronstein and Kopanskii \cite[Section 8]{BroK88}; cf. also Li
\cite[Proposition 2.5]{Li07}\textbf{.}
\end{remark}

For a linear semiflow $\Phi=(\theta,\phi)$ on $\mathcal{V}=B\times Y$ consider
two continuously varying, forward invariant subbundles with $\mathcal{V}%
=\mathcal{E}\oplus\mathcal{F}$ for which $\dim\mathcal{E}<\infty$. The
subbundles $\mathcal{E}$ and $\mathcal{F}$ are exponentially separated if
there exist constants $K,\gamma>0$ such that for all fibers $\mathcal{E}_{b}$
and $\mathcal{F}_{b},b\in B$,%
\[
\left\vert \Phi_{t}(b,\cdot)\right\vert _{\mathcal{F}_{b}}\leq Ke^{-\gamma
t}m(\Phi_{t}(b,\cdot))_{\mathcal{E}_{b}}\text{ for all }t>0,
\]
where $m$ is the minimum norm, $m(\Phi_{t}(b,\cdot))_{\mathcal{E}_{b}}%
:=\min\left\{  \left\vert \phi_{t}(b,y)\right\vert ,y\in\mathcal{E}%
_{b}\right\}  $.

The following theorem holds by \cite[Theorems A and B and Corollary
1.4]{BluL19}. The shorthand $1\leq i<N+1$ means that if $N=\infty$, then
$i\in\mathbb{N}$, and if $N<\infty$, then $1\leq i\leq N$.

\begin{theorem}
\label{Theorem_Selgrade1}Assume that $Y$ is a separable Banach space and that
$B$ is chain transitive for the base flow $\theta$. Let $\Phi$ be a linear
semiflow on $\mathcal{V}=B\times Y$ satisfying hypotheses (H1)--(H4) as above.
Then there is an at-most countable sequence $\{\mathcal{A}_{i}\}_{i=0}%
^{N},\allowbreak N\in\{0,1,\ldots\}\cup\{\infty\}$, of subsets of
$\mathbb{P}\mathcal{V}$ with $\mathcal{A}_{0}=\varnothing,\mathcal{A}%
_{i}\subset\mathcal{A}_{i+1}$ for $1\leq i<N$, with the following properties,
for any $1\leq i<N+1$:

(i) The set $\mathcal{A}_{i}$ is an asymptotically compact attractor for
$\mathbb{P}\Phi$.

(ii) The sequence $\{\mathcal{A}_{i}\}$ is the finest such collection in the
following sense: If $\mathcal{A}$ is any nonempty asymptotically compact
attractor for $\mathbb{P}\Phi$, then $\mathcal{A}=\mathcal{A}_{i}$ for some
$1\leq i<N+1$.

(iii) For the finite dimensional subbundles $\mathcal{V}_{i}^{+}%
=\mathbb{P}^{-1}\mathcal{A}_{i},i\in\{1,\ldots,N+1\}$, there are subbundles
$\mathcal{V}_{i}^{-}$ such that $\mathcal{V}=B\times Y=\mathcal{V}_{i}%
^{+}\oplus\mathcal{V}_{i}^{-}$ is an exponentially separated splitting of
$\mathcal{V}$.

(iv) The subbundles $\mathcal{V}_{i}:=\mathcal{V}_{i}^{+}\cap\mathcal{V}%
_{i-1}^{-}$ are finite dimensional, invariant subbundles of $\mathcal{V}$ such
that%
\[
\mathcal{V}_{i}^{+}=\mathcal{V}_{1}\oplus\cdots\oplus\mathcal{V}_{i}.
\]

(v) The sets $\mathcal{M}_{i}=\mathbb{P}\mathcal{V}_{i}$ are maximal chain
transitive for the projectivized flow $\mathbb{P}\Phi$ restricted to
$\mathbb{P}\mathcal{V}_{i}^{+}$.
\end{theorem}

The subbundles $\{\mathcal{V}_{i}\}_{i=1}^{N}$ are called the discrete
Selgrade decomposition of $\Phi$.

\begin{remark}
\label{Remark_order}For every $i$ with $1\leq i<N+1$ the sets $\mathcal{M}%
_{j}=\mathbb{P}\mathcal{V}_{j},1\leq j\leq i$, are the chain recurrent
components of the flow $\mathbb{P}\Phi$ restricted to $\mathbb{P}%
\mathcal{V}_{i}^{+}$. Since the $\mathcal{V}_{j}$ are finite dimensional they
are linearly ordered by $\mathcal{M}_{i}\preceq\mathcal{M}_{j}\,$for $i\leq j$
in the order of Morse sets on compact metric spaces $X$, cf. Colonius and
Kliemann \cite[Proposition B.2.8]{ColK00}): $\mathcal{M}_{i}\,\preceq
\mathcal{M}_{j}$ if there\ are $\mathcal{M}_{j_{0}}=\mathcal{M}_{i}%
,\mathcal{M}_{j_{1}},\ldots,\mathcal{M}_{j_{l}\text{ }}=\mathcal{M}_{j}$ with
$j\leq j_{1},\ldots,j_{l}=i$ and $x_{1},\allowbreak\ldots,x_{j_{l}}\in X$ with
$\omega^{\ast}(x_{k})\subset\mathcal{M}_{j_{k-1}}$ and $\omega(x_{k}%
)\subset\mathcal{M}_{j_{k}}$ for $k=1,\ldots,l$. Thus the subbundle
$\mathcal{V}_{1}$ is the most unstable subbundle (this is opposite to the
numbering in Colonius and Kliemann \cite[Theorem 5.1.4]{ColK00}).
\end{remark}

\section{Delay equations\label{Section3}}

In this section we consider linear control systems described by delay
equations of the form (\ref{D1}). The underlying field is either
$\mathbb{K}=\mathbb{R}$ or $\mathbb{K}=\mathbb{C}$.

The solutions $x(t)=\psi(t,r,f,0)$ of the homogeneous equation%
\begin{equation}
\dot{x}(t)=A_{0}x(t)+\sum_{i=1}^{p}A_{i}x(t-h_{i}) \label{hom}%
\end{equation}
satisfy (cf. Curtain and Zwart \cite[Theorem 3.3.1]{CurtZ})%
\[
x(t)=e^{A_{0}t}r+\sum_{i=1}^{p}\int_{0}^{t}e^{A_{0}(t-s)}A_{i}x(s-h_{i}%
)ds\text{ for }t\geq0.
\]
For $\tau>0$, there are constants $C_{\tau},D_{\tau}>0$ such that, for
$t\in\lbrack0,\tau]$, the following estimates hold:%
\begin{align}
\left\Vert x(t)\right\Vert ^{2}  &  \leq C_{\tau}\left[  \left\Vert
r\right\Vert ^{2}+\left\Vert f\right\Vert _{L^{2}([-h,0],\mathbb{K}^{n})}%
^{2}\right]  ,\label{CZ_3.67}\\
\int_{0}^{t}\left\Vert x(\tau)\right\Vert ^{2}d\tau &  \leq D_{\tau}\left[
\left\Vert r\right\Vert ^{2}+\left\Vert f\right\Vert _{L^{2}([-h,0],\mathbb{K}%
^{n})}^{2}\right]  . \label{CZ_3.68}%
\end{align}
This follows from the proof of \cite[Lemma 3.3.3]{CurtZ}. Arguing as in
\cite[Theorem 3.3.1]{CurtZ} one finds that the solution $x(t)=\psi(t,r,f,u)$
of the inhomogeneous equation (\ref{D1}) satisfies%
\begin{equation}
x(t)=e^{A_{0}t}r+\int_{0}^{t}e^{A_{0}(t-s)}\left[  \sum_{i=1}^{p}%
A_{i}x(s-h_{i})+\sum_{i=0}^{p}B_{i}u(s-h_{i})\right]  ds\text{ for }t\geq0.
\label{x_inhom}%
\end{equation}
We also recall the variation-of-parameters formula (cf. Hale \cite[Chapter 6,
Theorem 2.1 and Corollary 2.1 on pp. 143-145]{Hale}, Delfour \cite[Theorem
1.2]{Delfour77}). Let $X(t),t\geq0$, be the $n\times n$-matrix solution of%
\[
\frac{d}{dt}X(t)=A_{0}X(t)+\sum_{i=1}^{p}A_{i}X(t-h_{i})\text{ with
}X(t)=0\text{ for }t\in\lbrack-h,0),X(0)=I_{n}.
\]
Then the solution of the inhomogeneous equation (\ref{D1}) with initial value
$(r,f)\in\mathbb{R}^{n}\times L^{2}([-h,0],\mathbb{R}^{n})$ is given by%
\begin{equation}
\psi(t,r,f,u)=\psi(t,r,f,0)+\int_{0}^{t}X(t-s)\sum_{i=0}^{p}B_{i}%
u(s-h_{i})ds=\psi(t,r,f,0)+\psi(t,0,0,u). \label{VdP}%
\end{equation}
In particular, this shows that the solutions can be split into the homogeneous
and inhomogeneous parts.

Next we recall some facts on the description of these systems in the separable
Hilbert space $M_{2}:=M_{2}([-h,0],\mathbb{K}^{n})=\mathbb{K}^{n}\times
L^{2}([-h,0],\mathbb{K}^{n})$; cf. Curtain and Zwart \cite[Example
5.1.12]{CurtZ}. With the state $y(t)=\left(  x(t),x_{t}\right)  ^{\top}\in
M_{2}$ the homogeneous equation (\ref{hom}) can be reformulated as%
\begin{equation}
\dot{y}(t)=Ay(t),y(0)=y_{0}=\left(
\begin{array}
[c]{c}%
r\\
f
\end{array}
\right)  ,\,A\left(
\begin{array}
[c]{c}%
r\\
f
\end{array}
\right)  :=\left(
\begin{array}
[c]{c}%
A_{0}r+\sum_{i=1}^{p}A_{i}f(-h_{i})\\
\frac{df}{ds}(\cdot)
\end{array}
\right)  . \label{hom_delay}%
\end{equation}
Here $A$ is the infinitesimal generator of a strongly continuous semigroup
$T(t),t\geq0$, and the domain of definition of $A$ is%
\[
D(A)=\left\{  (r,f)^{\top}\left\vert f\in W^{1,2}([-h,0],\mathbb{K}^{n})\text{
and }f(0)=r\right.  \right\}  .
\]
The inhomogeneous equation (\ref{D1}) can be reformulated as%
\begin{equation}
\dot{y}(t)=Ay(t)+B\left(  u(t),u(t-h_{1}),\ldots,u(t-h_{p})\right)
,\,y(0)=y_{0}=\left(  r,f\right)  ^{\top}, \label{D2}%
\end{equation}
where%
\[
B:\mathbb{R}^{m(p+1)}\rightarrow M_{2},B(u_{0},u_{1},\ldots,u_{p}):=\left(
\sum\nolimits_{i=0}^{p}B_{i}u_{i},0\right)  ^{\top},
\]
is a bounded linear operator. The mild solution of (\ref{D2}) is
$y(t)=\varphi(t,y_{0},u)$ given by%
\begin{equation}
y(t)=T(t)y_{0}+\int_{0}^{t}T(t-s)B\left(  u(s),u(s-h_{1}),\ldots
,u(s-h_{p})\right)  ds. \label{mild}%
\end{equation}
It is related to the solution of (\ref{D1}) by $y(t)=\left(  x(t),x_{t}%
\right)  ^{\top}$. Note that%
\begin{equation}
\varphi(t,y_{0},u)=\varphi(t,y_{0},0)+\varphi(t,0,u)=T(t)y_{0}+\varphi
(t,0,u)\in M_{2},t\geq0. \label{phi_aff}%
\end{equation}
Since mild solutions are strongly continuous in $t$ the map $t\mapsto
\varphi(t,y_{0},u)$ is continuous for all $(y_{0},u)$ and $\varphi
(t,y_{0},u)\in D(A)$ for $t\geq h$. This follows since the solutions $x(t)$ of
(\ref{x_inhom}) are absolutely continuous with derivatives in $L^{2}$ on any
compact interval.

We note the following spectral properties of linear delay equations referring
to Curtain and Zwart \cite{CurtZ}, Manitius \cite{Mani80}, and also to
Diekmann, van Gils, Verduyn Lunel, and Walther \cite[Chapter 5]{DivGVLW}.

The infinitesimal generator $A$ has a pure point spectrum $\sigma(A)$
consisting of the (countably infinite or finitely many) zeros of
$\Delta(s)=\det(sI_{n}-\sum_{i=0}^{p}A_{i}e^{-h_{i}s})$. For every $\mu
\in\sigma(A)$ the generalized eigenspace is finite dimensional. The span of
the generalized eigenvectors of $A$ is dense in $M_{2}([-h,0],\mathbb{C}^{n})$
if and only if $\det A_{p}\not =0$ (cf. \cite[Corollary 5.5]{Mani80},
\cite[Theorem 3.4.4]{CurtZ}). The same is true for the span of the real
generalized eigenspaces $E(\mu),\mu\in\sigma(A)$, in $M_{2}([-h,0],\mathbb{R}%
^{n})$.

The linear delay equation is (uniformly) hyperbolic if $\sigma(A)\cap
\imath\mathbb{R=\varnothing}$, i.e. if $\Delta(\imath x)\not =0$ for all
$x\in\mathbb{R}$. A consequence of hyperbolicity is the following spectral
decomposition into a finite dimensional subspace $V^{+}$ and a stable subspace
$V^{-}$; cf. the discussion in Curtain and Zwart \cite[Theorem 8.2.5]{CurtZ}.

\begin{theorem}
\label{Theorem_hyperbolic_delay}If the linear semiflow $T(t),t\geq0$ is
hyperbolic, it admits a decomposition $M_{2}=V^{+}\oplus V^{-}$ into
$T(\cdot)$-invariant closed subspaces $V^{+}$ and $V^{-}$ with $\dim
V^{+}<\infty$ and constants $\alpha,K>0$ such for all $t\geq0$%
\[
\left\Vert T(t)y^{-}\right\Vert \leq Ke^{-\alpha t}\left\Vert y^{-}\right\Vert
\text{ for }y^{-}\in V^{-}\text{ and }\left\Vert T(t)y^{+}\right\Vert \geq
Ke^{\alpha t}\left\Vert y^{+}\right\Vert \text{ for }y^{+}\in V^{+}.
\]

\end{theorem}

\begin{remark}
Since $V^{+}$ is finite dimensional and $0$ is not in the spectrum of $T(t)$,
it follows that the restriction of $T(t)$ to $V^{+}$ is an isomorphism and we
can define $T(-t):=T(t)^{-1},t\geq0$. The condition above is equivalent to%
\[
\left\Vert T(t)z^{+}\right\Vert \leq K^{-1}e^{\alpha t}\left\Vert
z^{+}\right\Vert \text{ for }t\leq0\text{ and }z^{+}\in V^{+}.
\]
Indeed, one may write $y^{+}=T(-t)z^{+}$ with $z^{+}\in V^{+}$ and it follows
that%
\[
\left\Vert z^{+}\right\Vert \geq Ke^{\alpha t}\left\Vert T(-t)z^{+}\right\Vert
\text{ for }t\geq0.
\]

\end{remark}

The following result is a special case of Theorem \ref{Theorem_Selgrade1}.

\begin{theorem}
\label{Theorem_Selgrade_hom}Consider the linear semigroup $T(t),t\geq0$ and
assume that $T(t)$ is injective for $t\geq0$. Order the eigenvalues according
to their real parts $\operatorname{Re}\mu_{j}$ in decreasing order and define
for $i\in\mathbb{N}$ subspaces of $M_{2}$ by%
\[
V_{i}=\bigoplus_{\operatorname{Re}\mu_{j}=\lambda_{i}}E(\mu_{j}),~V_{i}%
^{+}=\bigoplus_{\operatorname{Re}\mu_{j}\geq\lambda_{i}}E(\mu_{j})\text{, and
}V_{i}^{-}=\overline{\bigoplus_{\operatorname{Re}\mu_{j}<\lambda_{i}}E(\mu
_{j})}.
\]
Then $V_{i}^{+}=V_{1}\oplus\cdots\oplus V_{i}$ and it follows that
$V_{i}=V_{i}^{+}\cap V_{i-1}^{-}$ are finite dimensional, forward invariant
subspaces of $M_{2}$ and $M_{2}=V_{i}^{+}\oplus V_{i}^{-}$ is an exponentially
separated splitting. Furthermore the following properties hold:

(i) For any $1\leq i<N^{0}:=\infty$ the set $A_{i}^{0}:=\mathbb{P}V_{i}^{+}$
is an asymptotically compact attractor for the projectivized semiflow
$\mathbb{P}T(\cdot)$.

(ii) Any nonempty asymptotically compact attractor for $\mathbb{P}T(\cdot)$
coincides with some $A_{i}^{0},i\in\mathbb{N}.$

(iii) For every $i\in\mathbb{N}$ the set $\mathcal{M}_{i}=\mathbb{P}V_{i}$ is
maximal chain transitive for the projectivized flow $\mathbb{P}T(\cdot)$
restricted to $\mathbb{P}V_{i}^{+}$.
\end{theorem}

\begin{proof}
The linear semigroup $T(t),t\geq0$, can be considered as a linear semiflow on
a vector bundle with trivial base space. Hypotheses (H1)-(H4) of Theorem
\ref{Theorem_Selgrade1} hold by the definitions and strong continuity of
$T(\cdot)$.
\end{proof}

\begin{remark}
For the system without control restriction and without delays in the control,
the reachable subspace from the origin is dense in $M_{2}([-h,0],\mathbb{C}%
^{n})$ if and only if $\mathrm{rank}\left[  \Delta(s),B_{0}\right]  =n$ for
all $s\in\mathbb{C}$ and $\mathrm{rank}\left[  A_{p},B_{0}\right]  =n$
(Curtain and Zwart \cite[Theorem 6.3.13]{CurtZ}).
\end{remark}

We note the following injectivity property of delay equations.

\begin{proposition}
\label{Proposition_injective1}Consider the linear delay control system
(\ref{D1}). Then for all $t\geq0$ the maps $T(t)$ are injective if and only if
the matrix $A_{p}$ is invertible.
\end{proposition}

\begin{proof}
If $A_{p}$ is not invertible, we may choose $f\in L^{2}([-h,0],\mathbb{R}%
^{n})$ with $0\not =f(s)\in\ker A_{p}$ for $s\in\lbrack-h,-h_{p-1})$ and
$f(s)=0$ for $s\in\lbrack-h_{p-1},0]$. Then the solutions satisfy
$\psi(t,r,f,u)=\psi(t,r,0,u)$ for $t>0$ and hence $T(t)$ is not injective.

Conversely, suppose that $A_{p}$ is invertible. Let $H=h-h_{p-1}=h_{p}%
-h_{p-1}$. Since for $t\geq0$ there are $k\in\{0,1,\ldots\}$ and $\tau
\in\lbrack0,H)$ with $t=\tau+kH$, it follows that $T(t)=T(\tau)\circ T(kH)$.
This shows that it suffices to prove injectivity of $T(\tau)$ for $\tau
\in(0,H]$. We have to prove that $\varphi(\tau,r,f,u)=\varphi(\tau,r^{\prime
},f^{\prime},u)$ implies $(r,f)=(r^{\prime},f^{\prime})$ or, equivalently,
that%
\begin{equation}
\psi(\tau+s,r,f,0)=0,s\in\lbrack-h,0]\text{, implies }r=0\text{ and }f=0.
\label{s}%
\end{equation}
Evaluation at $\tau+s=0$ shows that $r=0$. In order to show that $f=0$ in
$L^{2}([-h,0],\mathbb{R}^{n})$, note that, by (\ref{s}), $f(s)=0$ for
$s\in\lbrack\tau-h,0]$ holds.\ It remains to prove that $f(s)=0$ for (almost
all) $s\in\lbrack-h,\tau-h]$. We know that $x(t)=\psi(t,0,f,0)=0$ for
$t\in\lbrack0,\tau]$, hence it follows that for $t\in\lbrack0,\tau]$%
\[
f(t-h)=x(t-h)=A_{p}^{-1}\left[  \dot{x}(t)-\sum_{i=0}^{p-1}A_{i}%
x(t-h_{i})\right]  =-A_{p}^{-1}\sum_{i=1}^{p-1}A_{i}x(t-h_{i}).
\]
Since $\tau\in\lbrack0,H]=[0,h-h_{p-1}]$ it follows that $\tau-h\leq-h_{p-1}$.
Thus, it holds that. for all $i=1,\ldots,p-1$ and for $t\geq0$, that
$t-h_{i}\geq t-h_{p-1}\geq-h_{p-1}\geq\tau-h$ and hence $x(t-h_{i})=0$ for
$t-h\in\lbrack-h,\tau-h]$. It follows that%
\[
f(t-h)=x(t-h)=0\text{ for }t-h\in\lbrack-h,\tau-h],
\]
showing that $f(s)=0$ for $s\in\lbrack-h,\tau-h]$.
\end{proof}

\begin{remark}
Injectivity of the operators $T(t)$ is equivalent to the property that the
structural operator $F$ of the homogeneous delay differential equation
(\ref{hom}) has a trivial kernel; cf. Manitius \cite{Mani80}.
\end{remark}

\section{The control semiflow\label{Section4}}

This section constructs a continuous semiflow associated with the delay
control system (\ref{D1}) and analyzes some of its properties.

Endow the set $\mathcal{U}$ of controls with the following metric which is
compatible with the weak$^{\ast}$ topology on $L^{\infty}(\mathbb{R}%
,\mathbb{R}^{m})$:
\begin{equation}
d_{\mathcal{U}}(u,v)=\sum_{k=1}^{\infty}\frac{1}{2^{k}}\frac{\mid
\int_{\mathbb{R}}\left(  u(t)-v(t)\right)  ^{\top}z_{k}(t)dt\mid}{1+\mid
\int_{\mathbb{R}}\left(  u(t)-v(t)\right)  ^{\top}z_{k}(t)dt\mid},
\label{metric on U}%
\end{equation}
where $\left\{  z_{k},\;k\in\mathbb{N}\right\}  $ is a dense subset of
$L^{1}(\mathbb{R}\mathbf{,}\mathbb{R}^{m}).$ With this metric, $\mathcal{U}$
is a compact separable metric space and the right shift $\theta:\mathbb{R}%
\times\mathcal{U}\rightarrow\mathcal{U}:\theta_{t}u:=u(t+\cdot),t\in
\mathbb{R}$, is a continuous flow that is chain transitive on $\mathcal{U}$;
cf. Kawan \cite[Proposition 1.9]{Kawan13}. The map
\begin{equation}
\Phi:\mathbb{R}\times\mathcal{U}\times M_{2}\rightarrow\mathcal{U}\times
M_{2},\Phi_{t}\left(  u,y_{0}\right)  =\left(  \theta_{t}u,y(t)\right)
=\left(  \theta_{t}u,\varphi(t,y_{0},u)\right)  , \label{phi}%
\end{equation}
satisfies the semigroup properties $\Phi_{0}\left(  u,y_{0}\right)  =\left(
u,y_{0}\right)  $ and for $t,s\geq0$%
\begin{align*}
\Phi_{t+s}\left(  u,y_{0}\right)   &  =\left(  \theta_{t+s}u,\varphi
(t+s,y_{0},u)\right)  =\left(  \theta_{t}(\theta_{s}u),\varphi(t,\varphi
(s,y_{0},u),u(s+\cdot)\right) \\
&  =\Phi_{t}\circ\Phi_{s}\left(  u,y_{0}\right)  .
\end{align*}
We will show that $\Phi$ is a continuous semiflow, called a control semiflow.
Observe that $\Phi$ is not linear in the second argument, since the maps
$y_{0}\mapsto\varphi(t,y_{0},u)$ are affine; cf. (\ref{phi_aff}). The
homogeneous part of $\Phi$ is a product flow and we denote it by%
\begin{equation}
\Phi_{t}^{0}(u,y_{0})=(\theta_{t}u,\varphi(t,y_{0},0))=(\theta_{t}%
u,T(t)y_{0})\text{ for }t\geq0,u\in\mathcal{U},y_{0}\in M_{2}. \label{phi_0}%
\end{equation}
First we show properties of the homogeneous part $\Phi^{0}$.

\begin{proposition}
\label{Proposition_continuous}The homogeneous part $\Phi^{0}$ of the control
semiflow $\Phi:[0,\infty)\times\mathcal{U}\times M_{2}\rightarrow
\mathcal{U}\times M_{2}$ satisfies the following properties:

(i) The map $(t,u)\mapsto\Phi^{0}(t,u,\cdot):[0,\infty)\times\mathcal{U}%
\rightarrow\mathcal{U}\times\mathcal{L}(M_{2})$ is continuous with the strong
operator topology on the space $\mathcal{L}(M_{2})$ of bounded linear
operators on $M_{2}$.

(ii) For each fixed $t\geq0$, the map%
\[
\mathcal{U}\rightarrow\mathcal{U}\times\mathcal{L}(M_{2}):u\mapsto\Phi_{t}%
^{0}(u,\cdot)=(\theta_{t}u,T(t)),u\in\mathcal{U},
\]
is continuous with the operator norm topology on $\mathcal{L}(M_{2})$.

(iii) The semiflow $\Phi^{0}$ is continuous.
\end{proposition}

\begin{proof}
(i) Let $y\in M_{2}$. We have to show that, for a convergent sequence
$(t_{k},u^{k})\rightarrow(t_{0},u^{0})$, it follows that%
\[
\Phi_{t_{k}}^{0}(u^{k},y)=(\theta_{t_{k}}u^{k},T(t_{k})y))\rightarrow
\Phi_{t_{0}}^{0}(u^{0},y)=(\theta_{t_{0}}u^{0},T(t_{0})y).
\]
Convergence in the first component holds by continuity of the right shift
$\theta$, and $T(t_{k})y$ converges to $T(t^{0})y$ by strong continuity of the
semigroup $T(t),t\geq0$.\ 

(ii) Convergence in the first component follows as in (i) and, by
(\ref{phi_0}), $\left\Vert \Phi_{t}^{0}(u,\cdot)\right\Vert =\left\Vert
T(t)\right\Vert $ does not depend on the control $u$, hence continuity
trivially holds.

(iii) Due to assertions (i) and (ii) hypotheses (H1)-(H4) hold for $\Phi^{0}$,
except for injectivity of $\Phi(t,u,\cdot)$. Since Proposition
\ref{Proposition_cont1} does not need the injectivity assumption, this implies
that $\Phi^{0}$ is continuous.
\end{proof}

The following theorem establishes continuity of the affine control semiflow
$\Phi$.

\begin{theorem}
\label{Theorem_continuous}The control semiflow $\Phi:[0,\infty)\times
\mathcal{U}\times M_{2}\rightarrow\mathcal{U}\times M_{2}$ defined in
(\ref{phi}) satisfies the following properties:

(i) For every $y\in M_{2}$, the map $(t,u)\mapsto\Phi(t,u,y)=(\theta
_{t}u,\varphi(t,y,u)):[0,\infty)\times\mathcal{U}\rightarrow\mathcal{U}\times
M_{2}$ is continuous.

(ii) Let $t\geq0$. Then $u^{k}\rightarrow u^{0}$ in $\mathcal{U}$ implies that%
\[
\sup_{\left\Vert y\right\Vert \leq1}\left\Vert \varphi(t,y,u^{k}%
)-\varphi(t,y,u^{0})\right\Vert \rightarrow0.
\]

(iii) The semiflow $\Phi$ is continuous.
\end{theorem}

\begin{proof}
(i) Let $y\in M_{2}$. We have to show that for a convergent sequence
$(t_{k},u^{k})\rightarrow(t_{0},u^{0})$ it follows that%
\[
\Phi_{t_{k}}(u^{k},y)=(\theta_{t_{k}}u^{k},\varphi(t_{k},y,u^{k}%
))\rightarrow\Phi_{t_{0}}(u^{0},y)=(\theta_{t_{0}}u^{0},\varphi(t_{0}%
,y,u^{0})).
\]
Convergence in the first component holds by continuity of the right shift
$\theta$. Concerning the second component, formula (\ref{mild}) shows that,%
\begin{align*}
\varphi(t_{k},y,u^{k})  &  =T(t_{k})y+\int_{0}^{t_{k}}T(t_{k}-s)B\left(
u^{k}(s),u^{k}(s-h_{1}),\ldots,u^{k}(s-h_{p})\right)  ds\\
&  =\varphi(t_{k},y,0)+\varphi(t_{k},0,u^{k}).
\end{align*}
By Proposition \ref{Proposition_continuous}(i) it follows that%
\[
(\theta_{t_{k}}u^{k},\varphi(t_{k},y,0))\rightarrow\Phi_{t_{0}}^{0}%
(u^{0},y)=(\theta_{t_{0}}u^{0},\varphi(t_{0},y,0)).
\]
The other summand in the second component is $\varphi(t_{k},0,u^{k}%
)\allowbreak=(x^{k}(t_{k}),x_{t_{k}}^{k})$, where $x^{k}(t)=\psi
(t,0,0,u^{k}),t\geq-h$ is the solution of%
\[
\dot{x}^{k}(t)=A_{0}x^{k}(t)+\sum_{i=1}^{p}A_{i}x^{k}(t-h_{i})+\sum_{i=0}%
^{p}B_{i}u^{k}(t-h_{i}),t\geq0,
\]
with\ initial condition $x^{k}(t)=0$ and $u^{k}(t)=0$ for $t\in\lbrack-h,0]$.

Let $\tau>0$. We claim that $\psi(t,0,0,u^{k})\rightarrow\psi(t,0,0,u^{0})$
uniformly for $t\in\lbrack0,\tau]$.

The variation-of-parameters formula (\ref{VdP}) yields%
\[
\psi(t,0,0,u^{k})=\int_{0}^{t}X(t-s)\sum_{i=0}^{p}B_{i}u^{k}(s-h_{i})ds.
\]
Since $\Omega$ is compact there exists $c_{0}>0$ such that $\left\Vert
\psi(t,0,0,u^{k})\right\Vert \leq c_{0},t\in\lbrack0,\tau]$ for all $k$.
Furthermore, for $t_{1}<t_{2}$ in $[0,\tau]$,%
\begin{align*}
&  \left\Vert \psi(t_{2},0,0,u^{k})-\psi(t_{1},0,0,u^{k})\right\Vert \leq
\int_{t_{1}}^{t_{2}}\left[  \left\Vert X(t-s)\right\Vert \sum_{i=0}%
^{p}\left\Vert B_{i}\right\Vert \left\Vert u^{k}(s-h_{i})\right\Vert \right]
ds\\
&  \leq(t_{2}-t_{1})(p+1)\max_{s\in\lbrack0,\tau]}\left\Vert X(s)\right\Vert
\max_{i}\left\Vert B_{i}\right\Vert \max_{u\in\Omega}\left\Vert u\right\Vert .
\end{align*}
Thus, the functions $x^{k}(t)=\psi(t,0,0,u^{k}),t\in\lbrack0,\tau]$, are
equicontinuous and bounded and hence, by the Arzela-Ascoli theorem, a
uniformly converging subsequence exists.

Taking into account $u(s)=0$ for $s\leq0$, weak$^{\ast}$ convergence of
$u^{k}$ to $u^{0}$ implies that, for $t\in\lbrack0,\tau]$ and $i=0,1,\ldots
,p$,%
\begin{align*}
&  \int_{0}^{t}X(t-s)B_{i}u^{k}(s-h_{i})ds\\
&  =\int_{-h_{i}}^{t-h_{i}}X(t-s-h_{i})B_{i}u^{k}(s)ds=\int_{0}^{\infty}%
\chi_{\lbrack-h_{i},t-h_{i}]}(s)X(t-s-h_{i})B_{i}u^{k}(s)ds\\
&  \longrightarrow\int_{0}^{\infty}\chi_{\lbrack-h_{i},t-h_{i}]}%
(s)X(t-s-h_{i})B_{i}u^{0}(s)ds=\int_{0}^{t}X(t-s)B_{i}u^{0}(s-h_{i})ds.
\end{align*}
It follows that $x^{k}(t)=\psi(t,0,0,u^{k})$ converges to $x^{0}%
(t)=\psi(t,0,0,u^{0})$. By the preceding argument, this convergence is uniform
for $t\in\lbrack0,\tau]$.

We conclude that $x_{t_{k}}^{k}$ converges to $x_{t^{0}}^{0}$ in
$C([-h,0],\mathbb{R}^{n})$. This implies convergence of $(x^{k}(t_{k}%
),x_{t_{k}}^{k})$ to $(x^{0}(t^{0}),x_{t_{0}}^{0})$ in $M_{2}$ since the
embedding of $C([-h,0],\mathbb{R}^{n})$ into $M_{2}$ is continuous.

(ii) Let $t>0$ and consider $u^{k}\rightarrow u^{0}$ in $\mathcal{U}$. By
formula (\ref{phi_aff}) it follows that%
\[
\sup\left\{  \left\Vert \varphi(t,y,u^{k})-\varphi(t,y,u^{0})\right\Vert
\left\vert \left\Vert y\right\Vert \leq1\right.  \right\}  =\left\Vert
\varphi(t,0,u^{k})-\varphi(t,0,u^{0})\right\Vert .
\]
Thus, this does not depend on $y$, and assertion (i) implies that the right
hand side converges to $0$ for $k\rightarrow\infty$, as claimed.

(iii) For a sequence $(t_{k},u^{k},y_{k})\rightarrow(t_{0},u^{0},y_{0})$ in
$[0,\infty)\times\mathcal{U}\times M_{2}$, one obtains that%
\begin{align*}
&  \left\Vert \varphi(t_{k},y_{k},u^{k})-\varphi(t_{0},y_{0},u^{0})\right\Vert
\\
&  \leq\left\Vert \varphi(t_{k},y_{k},u^{k})-\varphi(t_{k},y_{0}%
,u^{k})\right\Vert +\left\Vert \varphi(t_{k},y_{0},u^{k})-\varphi(t_{0}%
,y_{0},u^{0})\right\Vert \\
&  =\left\Vert \varphi(t,y_{k},0)-\varphi(t,y_{0},0)\right\Vert +\left\Vert
\varphi(t_{k},0,u^{k})-\varphi(t_{0},0,u^{0})\right\Vert .
\end{align*}
The first summand equals $\left\Vert T(t)(y_{k}-y_{0})\right\Vert \leq$
$\sup_{k}\left\Vert T(t)\right\Vert \left\Vert y_{k}-y_{0}\right\Vert $ and
hence converges to $0$. The second summand converges to $0$ by (i).
\end{proof}

The following proposition characterizes injectivity of the semiflow associated
with linear delay control systems.

\begin{proposition}
\label{Proposition_injective2}Consider the linear delay control system
(\ref{D1}). Then for all $t>0$ the maps $\Phi_{t}$ on $\mathcal{U}\times
M_{2}$ are injective if and only if the matrix $A_{p}$ is invertible.
\end{proposition}

\begin{proof}
First we note that for $t>0$ the map $\Phi_{t}$ is injective if and only if
the maps $\Phi_{t}(u,\cdot),u\in\mathcal{U}$, are injective. In fact, suppose
that for all $u\in\mathcal{U}$ the map $\Phi_{t}(u,\cdot)$ is injective. Then
$\Phi_{t}(u,y)=\Phi_{t}(u^{\prime},y^{\prime})$ implies $u(t+\cdot)=u^{\prime
}\left(  t+\cdot\right)  $, hence $u=u^{\prime}$ implying $y=y^{\prime}$. This
shows that $\Phi_{t}$ is injective on $\mathcal{U}\times M_{2}$. The converse
holds trivially. Furthermore, the map $\Phi_{t}(u,\cdot)$ is injective if and
only if $\Phi_{t}(0,\cdot)=T(t)$ is injective. By Proposition
\ref{Proposition_injective1} this holds if and only if $\det A_{p}\not =0$.
\end{proof}

Theorem \ref{Theorem_Selgrade_hom} on spectral theory of delay equations
entails the following consequences for the homogeneous part $\Phi^{0}$ of
$\Phi$.

\begin{proposition}
\label{Proposition_Selgrade_hom}Consider the linear flow $\Phi^{0}$ on
$\mathcal{U}\times M_{2}$ given by (\ref{phi_0}) and assume that $\det
A_{p}\not =0$. Then the sequence $\{\mathcal{A}_{i}^{0}\}_{i=0}^{\mathbb{N}%
}:=\{\mathcal{U}\times A_{i}^{0}\}_{i=0}^{\mathbb{N}}$ of subsets of
$\mathcal{U}\times\mathbb{P}M_{2}$ has the following properties.

(i) For any $i\in\mathbb{N}$ the set $\mathcal{A}_{i}^{0}$ is an
asymptotically compact attractor for the projectivized semiflow $\mathbb{P}%
\Phi^{0}$ on $\mathcal{U}\times\mathbb{P}M_{2}$.

(ii) If $\mathcal{A}^{0}$ is any nonempty asymptotically compact attractor for
$\mathbb{P}\Phi^{0}$, then $\mathcal{A}^{0}=\mathcal{A}_{i}^{0}$ for some
$i\in\mathbb{N}$.

(iii) For $\mathcal{V}_{i}^{0,+}:=\mathbb{P}^{-1}\mathcal{A}_{i}%
^{0}=\mathcal{U}\times\mathbb{P}^{-1}A_{i}^{0},i\in\mathbb{N}$ the subbundles
$\mathcal{V}_{i}^{0,-}:=\mathcal{U}\times V_{i}^{-}$ yield an exponentially
separated splitting $\mathcal{U}\times M_{2}=\mathcal{V}_{i}^{0,+}%
\oplus\mathcal{V}_{i}^{0,-}$.

(iv) The subbundles $\mathcal{V}_{i}^{0}:=\mathcal{V}_{i}^{0,+}\cap
\mathcal{V}_{i-1}^{0,-}=\mathcal{U}\times\left(  V_{i}^{+}\cap V_{i-1}%
^{-}\right)  $ are finite dimensional, invariant subbundles of $\mathcal{U}%
\times M_{2}$ such that%
\[
\mathcal{V}_{i}^{0,+}=\mathcal{V}_{1}^{0}\oplus\cdots\oplus\mathcal{V}_{i}%
^{0}.
\]

(v) The sets $\mathcal{M}_{i}^{0}=\mathbb{P}\mathcal{V}_{i}^{0}$ are maximal
chain transitive for the projectivized flow $\mathbb{P}\Phi^{0}$ restricted to
$\mathbb{P}\mathcal{V}_{i}^{0,+}$.
\end{proposition}

\begin{proof}
By Proposition \ref{Proposition_injective2}, the map $\Phi_{t}^{0}(u,\cdot)$
on $M_{2}$ is injective for every $u\in\mathcal{U}$ since $\det A_{p}\not =0$.
For the proof of assertion (i), observe that by Theorem
\ref{Theorem_Selgrade_hom}(ii) the sets $A_{i}^{0}$ are asymptotically compact
attractors for $\mathbb{P}T(\cdot)$. Hence, for $\varepsilon>0$ small enough,
the neighborhood $B_{\varepsilon}(A_{i}^{0})$ of $A_{i}^{0}$ satisfies

(a) $\overline{\mathbb{P}T([S,\infty)\times B_{\varepsilon}(A_{i}^{0}%
))}\subset B_{\varepsilon}(A_{i}^{0})$ and $A_{i}^{0}=\omega(B_{\varepsilon
}(A_{i}^{0})$;

(b) for any sequence $t_{k}\rightarrow\infty$ and any sequence of points
$y_{k}\in B_{\varepsilon}(A_{i}^{0})$, it follows that $\mathbb{P}%
T(t_{k})(y_{k}),k\in\mathbb{N}$ has a convergent subsequence.

It is easily seen that the neighborhood $\mathcal{U}\times B_{\varepsilon
}(A_{i}^{0})$ satisfies the analogous conditions for $\mathbb{P}\Phi^{0}$
instead of $T(\cdot)$. This proves (i). Assertion (ii) is analogously derived
from the corresponding property of the semiflow $T(\cdot)$. Assertions (iii)
and (iv) are easily seen using the definitions.

Together we have shown that one obtains a discrete Selgrade decomposition of
$\mathbb{P}\Phi^{0}$ restricted to $\mathcal{U}\times M_{2}$. Thus Theorem
\ref{Theorem_Selgrade1}(v) implies assertion (v).
\end{proof}

\section{The chain control set\label{Section5}}

In this section we define chain control sets and show that they correspond to
maximal chain transitive subsets of the control semiflow. Furthermore, we
prove that there exists a unique chain control set $E$ in $M_{2}$. Throughout
the rest of this paper, we assume that the matrix $A_{p}$ is invertible and
hence, by Proposition \ref{Proposition_injective2}, the maps $\Phi_{t}$ are
injective on $\mathcal{U}\times M_{2}$.

As in Section \ref{Section2} we write $\Phi_{-t}(u,y)\in\mathcal{U}\times Y$
for the pre-image $\left(  \Phi_{t}\right)  ^{-1}(u,y)$. Since the shift
$\theta_{t}$ is defined for all $t\in\mathbb{R}$, the pre-image exists if and
only if the pre-image of $y$ under the map $\varphi(t,\cdot,u)$ denoted by
$\varphi(-t,y,u)$ exists. When we write $\Phi_{-t}(u,y)$ or $\varphi(-t,y,u)$
we tacitly suppose that these pre-images exist.

Fix $y,z\in M_{2}$ and let $\varepsilon,\tau>0$. A controlled $(\varepsilon
,\tau)$\textit{-}chain $\zeta$ from $y$ to $z$ is given by $q\in
\mathbb{N},\ y_{0}=y,y_{1},\ldots,y_{q}=z$ in $M_{2},\ u_{0},\ldots,u_{q-1}%
\in\mathcal{U}$, and $\tau_{0},\ldots,\tau_{q-1}\geq\tau$ with
\[
\left\Vert \varphi(\tau_{j},y_{j},u_{j})-y_{j+1}\right\Vert <\varepsilon\text{
}\,\text{for }j=0,\ldots,q-1.
\]

\begin{definition}
\label{Definition_chain}A nonvoid set $E\subset M_{2}$ is chain controllable,
if for all $y,z\in E$ and all $\varepsilon,\tau>0$ there are controlled
$(\varepsilon,\tau)$-chains from $y$ to $z$. A set $E\subset M_{2}$ is weakly
invariant, if for every $y\in E$ there exists $u\in\mathcal{U}$ such that for
all $t\in\mathbb{R}$ one has $\varphi(t,y,u)\in E$. A weakly invariant chain
controllable set $E\subset M_{2}$ is a chain control set if it is maximal
(with respect to set inclusion) with these properties.
\end{definition}

For finite dimensional systems, the definition of chain control sets above
coincides with the standard definition of chain control sets (cf. Colonius and
Kliemann \cite{ColK00}, Kawan \cite{Kawan13}).

Chain control sets can be characterized using the control semiflow $\Phi$ on
$\mathcal{U}\times M_{2}$. The following theorem and its proof are
generalizations of an analogous result for finite dimensional systems; cf.
\cite[Theorem 4.3.11]{ColK00}. Recall from Section \ref{Section2} that any
invariant set for a semiflow consists of points defining entire solutions.

\begin{theorem}
\label{Theorem_equivalence}Consider the linear delay control system given by
(\ref{D1}) and assume that the matrix $A_{p}$ is invertible.

(i) If $E\subset M_{2}$ is a weakly invariant chain controllable set, then the
lift%
\[
\mathcal{E}:=\{(u,y)\in\mathcal{U}\times M_{2}\left\vert \forall
t\in\mathbb{R}:\varphi(t,y,u)\in E\right.  \}
\]
is an invariant chain transitive set for the control semiflow $\Phi$ and, in
particular, $\mathcal{E}$ is contained in the entire chain recurrent set
$\mathcal{R}^{\#}$ of $\Phi$.

(ii) Conversely, let $\mathcal{E}\subset\mathcal{U}\times M_{2}$ be an
invariant chain transitive set for $\Phi$. Then%
\[
\pi_{M_{2}}\mathcal{E}:=\left\{  y\in M_{2}\left\vert \exists u\in
\mathcal{U}:(u,y)\in\mathcal{E}\right.  \right\}
\]
is a weakly invariant chain controllable set.

(iii) For a chain control set $E$ the set $\mathcal{E}$ is a maximal invariant
chain transitive set of $\Phi$, and conversely, if $\mathcal{E}$ is a maximal
invariant chain transitive set, then $\pi_{M_{2}}\mathcal{E}$ is a chain
control set.
\end{theorem}

\begin{proof}
(i) Let $(u,y),(v,z)\in\mathcal{E}$ and pick $\varepsilon,\,\tau>0.$ Recall
the definition of the metric $d_{\mathcal{U}}$ on $\mathcal{U}$ in
(\ref{metric on U}) and choose $K\in\mathbb{N}$ large enough such that
$\sum_{k=K+1}^{\infty}2^{-k}<\frac{\varepsilon}{2}$. For $z_{1},\ldots
,z_{K}\in L^{1}(\mathbb{R},\mathbb{R}^{m})$ we may take $S$ large enough such
that for all $i$%
\[
\int_{\mathbb{R}\setminus\left[  -S,S\right]  }\left\Vert z_{i}(\tau
)\right\Vert \,d\tau<\frac{\varepsilon}{2\mathrm{\,diam}\,U}.
\]
Chain controllability from $\varphi(2\tau,y,u)\in E$ to $\varphi(-\tau,z,v)\in
E$ yields the existence of $q\in\mathbb{N}$ and $y_{0},\ldots,y_{q}\in
M_{2},\,u_{0},\ldots,u_{q-1}\in\mathcal{U},\,\tau_{0},\ldots,\tau_{q-1}%
\geq\tau$ with $y_{0}=\varphi(2\tau,y,u),\,y_{q}=\varphi(-\tau,z,v)$, and
\[
\left\Vert \varphi(\tau_{j},y_{j},u_{j})-y_{j+1}\right\Vert <\varepsilon
\text{\ for }j=0,\ldots,q-1.
\]
We now construct an $(\varepsilon,\tau)$-chain from $(u,y)$ to $(v,z)$ in the
following way. Define
\[%
\begin{array}
[c]{cccc}%
\tau_{-2}=\tau, & y_{-2}=y, & v_{-2}=u, & \\
\tau_{-1}=\tau, & y_{-1}=\varphi(\tau,y,u), & v_{-1}(\tau)= & \left\{
\begin{array}
[c]{cc}%
u(\tau_{-2}+t) & \text{for }t\leq\tau_{-1}\\
u_{0}(t-\tau_{-1}) & \text{for }t>\tau_{-1}%
\end{array}
\right.
\end{array}
\]
and let the times $\tau_{0},\ldots,\tau_{q-1}$ and the points $y_{0}%
,...,y_{q}$ be as given earlier; furthermore, set $\tau_{q}=\tau
,y_{q+1}=z,v_{q+1}=v$, and define, for $j=0,\ldots,q-2$,%
\begin{align*}
v_{j}(t)  &  =\left\{
\begin{array}
[c]{lcl}%
v_{j-1}(\tau_{j-1}+t) & \text{for} & t\leq0\\
u_{j}(t) & \text{for} & 0<t<\tau_{j}\\
u_{j+1}(t-\tau_{j}) & \text{for} & t>\tau_{j},
\end{array}
\right. \\
v_{q-1}(t)  &  =\left\{
\begin{array}
[c]{lll}%
v_{q-2}(\tau_{q-2}+t) & \text{for} & t\leq0\\
u_{q-1}(t) & \text{for} & 0<t\leq\tau_{q-1}\\
v(t-\tau_{q-1}-\tau) & \text{for} & t>\tau_{q-1},
\end{array}
\right. \\
v_{q}(t)  &  =\left\{
\begin{array}
[c]{ll}%
v_{q-1}(\tau_{q-1}+t) & \text{for }t\leq0\\
v(t-\tau) & \text{for }t>0.
\end{array}
\right.
\end{align*}
It is easily seen that
\[
(v_{-2},y_{-2}),\,(v_{-1},y_{-1}),\ldots,(v_{q+1},y_{q+1})\text{ and }%
\tau_{-2},\,\tau_{-1},\ldots,\tau_{q}\geq\tau
\]
yield an $(\varepsilon,\tau)$-chain from $(u,y)$ to $(v,z)$ provided that
$d_{\mathcal{U}}(v_{j}(\tau_{j}+\cdot),v_{j+1})<\varepsilon$ for
$j=-2,-1,\ldots,q$. By choice of $S$ and $K$ one has for all $w_{1},\,w_{2}%
\in\mathcal{U}$ that the distance $d_{\mathcal{U}}(w_{1},w_{2})$ is bounded by%
\begin{align*}
&  \frac{\varepsilon}{2}+\sum_{k=1}^{K}\frac{1}{2^{k}}\left[  \left\Vert
\int_{\mathbb{R}\setminus\left[  -S,S\right]  }\left(  w_{1}(t)-w_{2}%
(t)\right)  ^{\top}z_{k}(t)\,dt\right\Vert +\left\Vert \int_{-S}^{S}\left(
w_{1}(t)-w_{2}(t)\right)  ^{\top}z_{k}(t)\,dt\right\Vert \right] \\
&  <\varepsilon+\max_{k=1,\ldots,K}\int_{-S}^{S}\left\Vert w_{1}%
(t)-w_{2}(t)\right\Vert \,\left\Vert z_{k}(t)\right\Vert \,dt.
\end{align*}
Hence it suffices to show that for all considered pairs of control functions
the integrands vanish. This is immediate from the definition of $v_{j}%
,\,j=-2,\ldots,q+1$.

(ii) Let $\mathcal{E}$ be an invariant chain transitive set in $\mathcal{U}%
\times M_{2}$. For $y\in\pi_{M_{2}}\mathcal{E}$ there exists $u\in\mathcal{U}$
such that $\varphi(t,y,u)\in\pi_{M_{2}}\mathcal{E}$ for all $t\in\mathbb{R}$
by invariance. Now let $y,z\in\pi_{M}\mathcal{E}$ and choose $\varepsilon
,\tau>0$. Then by chain transitivity of $\mathcal{E}$ we can choose
$y_{j},u_{j},\tau_{j}$ such that the corresponding trajectories satisfy the
required conditions.

(iii) It is clear that, for a chain control set $E$, the set $\mathcal{E}$ is
a maximal invariant chain transitive set. Conversely, for a maximal invariant
chain transitive set $\mathcal{E}$ the projection $\pi_{M_{2}}\mathcal{E}$ to
$M_{2}$ is chain controllable and weakly invariant. Since the maximality
property of $\pi_{M_{2}}\mathcal{E}$ is clear, the assertion follows by (ii).
\end{proof}

\begin{corollary}
Assume that the matrix $A_{p}$ is invertible. Then the following assertions
are equivalent:

(i) The entire chain recurrent set $\mathcal{R}^{\#}$ of the semiflow $\Phi$
is chain transitive.

(ii) The set $\mathcal{R}^{\#}$ is the lift $\mathcal{E}$ of a chain control
set $E.$

(iii) There is a single chain control set $E$.
\end{corollary}

\begin{proof}
Suppose that (i) holds. Then, by Theorem \ref{Theorem_equivalence}(iii) the
projection to $M_{2}$ is a chain control set $E$. By Theorem
\ref{Theorem_equivalence}(iii) it follows that the lift $\mathcal{E}$ of $E$ is

maximal invariant chain transitive set and hence coincides with $\mathcal{R}%
^{\#}$. This implies (iii) since the lift of any chain control set is
contained in $\mathcal{R}^{\#}$. Finally, if $E$ is unique it follows that
$\mathcal{R}^{\#}$ coincides with the lift of $E$.
\end{proof}

The following theorem establishes the announced uniqueness of the chain
control set. While this result generalizes the finite dimensional case
(Colonius and Santana \cite[Theorem 29]{ColS24}), step 3 in the proof is
different since here we cannot argue with time reversal. For the convenience
of the reader we also write down steps 1 and 2 in the present setting.

\begin{theorem}
\label{Theorem_ccs1}Consider the linear delay control system given by
(\ref{D1}) and assume that the matrix $A_{p}$ is invertible. Then there exists
a unique chain control set $E$ in $M_{2}$.
\end{theorem}

\begin{proof}
First note that for $u\equiv0$ the origin $0\in M_{2}$ is an equilibrium,
hence $\{0\}$ is a weakly invariant chain controllable set. Define $E$ as the
union of all weakly invariant chain controllable sets containing $\{0\}$. Then
$E$ is a weakly invariant chain controllable set and certainly it is maximal
with these properties, hence it is a chain control set. It remains to prove uniqueness.

Observe that the trajectories $y(t)=\varphi(t,y_{0},u),t\in\mathbb{R}$, of
(\ref{mild}) satisfy, for $\alpha\in(0,1)$,%
\begin{align}
\alpha\varphi(t,y_{0},u)  &  =\alpha T(t)y_{0}+\alpha\int_{0}^{t}%
T(t-s)B\left(  u(s),u(s-h_{1}),\ldots,u(s-h_{p})\right)  ds\nonumber\\
&  =T(t)\alpha y_{0}+\int_{0}^{t}T(t-s)B\left(  \alpha u(s),\alpha
u(s-h_{1}),\ldots,\alpha u(s-h_{p})\right)  ds\nonumber\\
&  =\varphi(t,\alpha y_{0},\alpha u). \label{alpha1}%
\end{align}
Here, $\varphi(\cdot,\alpha y_{0},\alpha u)$ is a trajectory of (\ref{mild})
since $\Omega$ is a convex neighborhood of $0\in\mathbb{R}^{m}$ implying that
the controls $\alpha u$ are in $\mathcal{U}$.

Suppose that $E^{\prime}$ is any chain control set and let $y\in E^{\prime}$.
First we will construct controlled $(\varepsilon,\tau)$-chains from $y$ to
$0\in E$.

\textbf{Step 1:} There is a controlled $(\varepsilon,\tau)$-chain from $y$ to
$\alpha y$ for some $\alpha\in(0,1)$.

For the proof consider a controlled $(\varepsilon/2,\tau)$-chain $\zeta$ from
$y$ to $y$ given by $y_{0}=y,y_{1},\ldots,y_{q}=y,\,u_{0},\ldots,u_{q-1}%
\in\mathcal{U}$, and $\tau_{0},\ldots,\tau_{q-1}\geq\tau$ with%
\[
\left\Vert \varphi(\tau_{i},y_{i},u_{i})-y_{i+1}\right\Vert <\varepsilon
/2\text{ for }i=0,\ldots,q-1.
\]
Let $\alpha\in(0,1)$ with $(1-\alpha)\left\Vert y\right\Vert <\varepsilon/2$.
It follows that%
\[
\left\Vert \varphi(\tau_{q-1},y_{q-1},u_{q-1})-\alpha y_{q}\right\Vert
\leq\left\Vert \varphi(\tau_{q-1},y_{q-1},u_{q-1})-y\right\Vert +\left\Vert
y-\alpha y\right\Vert <\varepsilon.
\]
This defines a controlled $(\varepsilon,\tau)$-chain $\zeta^{(1)}$ from $y$ to
$\alpha y$.

\textbf{Step 2: }Replacing $y_{i}$ by $\alpha y_{i}$ and $u_{i}$ by $\alpha
u_{i}$ we get by (\ref{alpha1})%
\begin{align*}
\left\Vert \varphi(\tau_{i},\alpha y_{i},\alpha u_{i})-\alpha y_{i+1}%
\right\Vert  &  =\alpha\left\Vert \varphi(\tau_{i},y_{i},u_{i})-y_{i+1}%
\right\Vert <\varepsilon/2\text{ for }i=0,\ldots,q-1,\\
\left\Vert \varphi(\tau_{q-1},\alpha y_{q-1},\alpha u_{q-1})-\alpha
^{2}y\right\Vert  &  \leq\left\Vert \varphi(\tau_{q-1},\alpha y_{q-1},\alpha
u_{q-1})-\alpha y\right\Vert +\left\Vert \alpha y-\alpha^{2}y\right\Vert
<\varepsilon.
\end{align*}
This defines a controlled $(\varepsilon,\tau)$-chain $\zeta^{(2)}$ from
$\alpha y$ to $\alpha^{2}y$. The concatenation of $\zeta^{(2)}$ and
$\zeta^{(1)}$ yields a controlled $(\varepsilon,\tau)$-chain $\zeta^{(2)}%
\circ\zeta^{(1)}$ from $y$ to $\alpha^{2}y$.

Repeating this construction, we find that the concatenation $\zeta^{(k)}%
\circ\cdots\circ\zeta^{(1)}$ is a controlled $(\varepsilon,\tau)$-chain from
$y\in E^{\prime}$ to $\alpha^{k}y$. Since $\alpha^{k}\rightarrow0$ for
$k\rightarrow\infty$, we can take $k\in\mathbb{N}$ large enough, such that the
last piece of the chain $\zeta^{(k)}$ satisfies%
\[
\left\Vert \varphi(\tau_{q-1},a^{k}y_{q-1},a^{k}u_{q-1})\right\Vert
<\varepsilon.
\]
Thus we may take $0\in E$ as the final point of this controlled $(\varepsilon
,\tau)$-chain showing that the concatenation $\zeta^{(k)}\circ\cdots\circ
\zeta^{(1)}$ defines a controlled $(\varepsilon,\tau)$-chain from $y\in
E^{\prime}$ to $0\in E$.

\textbf{Step 3:} Next we construct controlled chains from $0$ to $y\in
E^{\prime}$.

Consider a controlled $(\varepsilon,\tau)$-chain from $y$ to $y$ given by
$y_{0}=y,y_{1},\ldots,y_{q}=y,\allowbreak\,u_{0},\ldots,u_{q-1}\in\mathcal{U}%
$, and $\tau_{0},\ldots,\tau_{q-1}\geq\tau$ with%
\[
\left\Vert \varphi(\tau_{i},y_{i},u_{i})-y_{i+1}\right\Vert <\varepsilon\text{
for }i=0,\ldots,q-1.
\]
For every $\alpha\in(0,1)$ formula (\ref{alpha1}) shows that $\alpha
y_{0}=\alpha y,\alpha y_{1},\ldots,\alpha y_{q}=\alpha y,\,\allowbreak\alpha
u_{0},\ldots,\allowbreak\alpha u_{q-1}\in\mathcal{U}$, and $\tau_{0}%
,\ldots,\tau_{q-1}\geq\tau$ define a controlled $(\alpha\varepsilon,\tau)$-
chain from $\alpha y$ to $\alpha y$ with%
\[
\left\Vert \varphi(\tau_{i},\alpha y_{i},\alpha u_{i})-\alpha y_{i+1}%
\right\Vert <\alpha\varepsilon\text{ for }i=0,\ldots,q-1.
\]
Let $\alpha\in(0,\varepsilon)$ be small enough such that $\alpha\left\Vert
y\right\Vert <\varepsilon$. Then we may add a segment $\varphi(t,0,0)=0,t\in
\lbrack0,\tau]$, at the beginning to obtain a controlled $(\alpha
\varepsilon,\tau)$-chain from $0$ to $\alpha y$. Furthermore, we find%
\begin{align*}
\left\Vert \varphi(\tau_{q-1},\alpha y_{q-1},\alpha u_{q-1})-(\alpha
+\varepsilon)y\right\Vert  &  \leq\left\Vert \varphi(\tau_{q-1},\alpha
y_{q-1},\alpha u_{q-1})-\alpha y\right\Vert +\varepsilon\left\Vert
y\right\Vert \\
&  \leq\alpha\varepsilon+\varepsilon\left\Vert y\right\Vert <(\varepsilon
+\left\Vert y\right\Vert )\varepsilon.
\end{align*}
Taking $\varepsilon\in(0,1)$ we have constructed a controlled $((1+\varepsilon
)\left\Vert y\right\Vert ,\tau)$-chain $\zeta^{(1)}$ from $0$ to
$(\alpha+\varepsilon)y$.

Now construct a controlled $(2\varepsilon,\tau)$-chain $\zeta^{(2)}$ from
$(\alpha+\varepsilon)y$ to $(2\alpha+\varepsilon)y$. By (\ref{alpha1}),%
\[
\left\Vert \varphi(\tau_{i},(\alpha+\varepsilon)y_{i},(\alpha+\varepsilon
)u_{i})-(\alpha+\varepsilon)y_{i+1}\right\Vert <(\alpha+\varepsilon
)\varepsilon\text{ for }i=0,\ldots,q-1,
\]
and, since $\alpha\left\Vert y\right\Vert <\varepsilon$,%
\begin{align*}
&  \left\Vert \varphi(\tau_{q-1},(\alpha+\varepsilon)y_{q-1},(\alpha
+\varepsilon)u_{q-1})-(2\alpha+\varepsilon)y\right\Vert \\
&  \leq\left\Vert \varphi(\tau_{q-1},(\alpha+\varepsilon)y_{q-1}%
,(\alpha+\varepsilon)u_{q-1})-(\alpha+\varepsilon)y\right\Vert +\alpha
\left\Vert y\right\Vert \\
&  <(\alpha+\varepsilon)\varepsilon+\varepsilon.
\end{align*}
For $\alpha+\varepsilon<1$, it follows that $(\alpha+\varepsilon
)\varepsilon+\varepsilon\leq2\varepsilon$ and hence this is a controlled
$(2\varepsilon,\tau)$-chain from $(\alpha+\varepsilon)y$ to $(2\alpha
+\varepsilon)y$.

Next construct a controlled $(2\varepsilon,\tau)$-chain $\zeta^{(3)}$ starting
in $(2\alpha+\varepsilon)y$: By (\ref{alpha1}),%
\[
\left\Vert \varphi(\tau_{i},(2\alpha+\varepsilon)y_{i},(2\alpha+\varepsilon
)u_{i})-(2\alpha+\varepsilon)y_{i+1}\right\Vert <(2\alpha+\varepsilon
)\varepsilon\text{ for }i=0,\ldots,q-1
\]
and%
\begin{align*}
&  \left\Vert \varphi(\tau_{q-1},(2\alpha+\varepsilon)y_{q-1},(2\alpha
+\varepsilon)u_{q-1})-(3\alpha+\varepsilon)y\right\Vert \\
&  \leq\left\Vert \varphi(\tau_{q-1},(2\alpha+\varepsilon)y_{q-1}%
,(2\alpha+\varepsilon)u_{q-1})-(2\alpha+\varepsilon)y\right\Vert
+\alpha\left\Vert y\right\Vert \\
&  <(2\alpha+\varepsilon)\varepsilon+\varepsilon.
\end{align*}
For $2\alpha+\varepsilon<1$, this defines a controlled $(2\varepsilon,\tau
)$-chain from $(2\alpha+\varepsilon)y$ to $(3\alpha+\varepsilon)y$.

As long as $j\alpha+\varepsilon<1$, we can proceed in this way to obtain
controlled $(2\varepsilon,\tau)$-chains $\zeta^{(j+1)}$ from $(j\alpha
+\varepsilon)y$ to $((j+1)\alpha+\varepsilon)y$ satisfying%
\[
\left\Vert \varphi(\tau_{i},(j\alpha+\varepsilon)y_{i},(j\alpha+\varepsilon
)u_{i})-(j\alpha+\varepsilon)y_{i+1}\right\Vert <(j\alpha+\varepsilon
)\varepsilon<\varepsilon\text{ for }i=0,\ldots,q-1,
\]
with%
\begin{align*}
&  \left\Vert \varphi(\tau_{q-1},(j\alpha+\varepsilon)y_{q-1},(j\alpha
+\varepsilon)u_{q-1})-((j+1)\alpha+\varepsilon)y\right\Vert \\
&  \leq\left\Vert \varphi(\tau_{q-1},(j\alpha+\varepsilon)y_{q-1}%
,(j\alpha+\varepsilon)u_{q-1})-(j\alpha+\varepsilon)y\right\Vert
+\alpha\left\Vert y\right\Vert \\
&  \leq(j\alpha+\varepsilon)\varepsilon+\alpha\left\Vert y\right\Vert
<2\varepsilon.
\end{align*}
When for $j=k$ we arrive at $k\alpha+\varepsilon<1$ and $(k+1)\alpha
+\varepsilon\geq1$, we find%
\begin{equation}
\varepsilon>\varepsilon+k\alpha+\varepsilon-1>(k+1)\alpha+\varepsilon-1\geq0.
\label{N+1}%
\end{equation}
Thus we get for $i=0,\ldots,q-1$%
\[
\left\Vert \varphi(\tau_{i},(k\alpha+\varepsilon)y_{i},(k\alpha+\varepsilon
)u_{i})-(k\alpha+\varepsilon)y_{i+1}\right\Vert <(k\alpha+\varepsilon
)\varepsilon<\varepsilon,
\]
and, by (\ref{N+1}),%
\begin{align*}
&  \left\Vert \varphi(\tau_{q-1},(k\alpha+\varepsilon)y_{q-1},(k\alpha
+\varepsilon)u_{q-1})-y\right\Vert \\
&  \leq\left\Vert \varphi(\tau_{q-1},(k\alpha+\varepsilon)y_{q-1}%
,(k\alpha+\varepsilon)u_{q-1})-((k+1)\alpha+\varepsilon)y\right\Vert \\
&  \qquad+\left\Vert ((k+1)\alpha+\varepsilon)y-y\right\Vert \\
&  <\left\Vert \varphi(\tau_{q-1},(k\alpha+\varepsilon)y_{q-1},(k\alpha
+\varepsilon)u_{q-1})-(k\alpha+\varepsilon)y\right\Vert +\varepsilon\left\Vert
y\right\Vert \\
&  \qquad+\left\Vert ((k+1)\alpha+\varepsilon-1)y\right\Vert \\
&  <(k\alpha+\varepsilon)\varepsilon+\varepsilon\left\Vert y\right\Vert
+((k+1)\alpha+\varepsilon-1)\left\Vert y\right\Vert \\
&  <\varepsilon+\varepsilon\left\Vert y\right\Vert +\varepsilon\left\Vert
y\right\Vert <\varepsilon+2\varepsilon\left\Vert y\right\Vert .
\end{align*}
Thus this defines a controlled $((1+2\left\Vert y\right\Vert )\varepsilon
,\tau)$-chain $\zeta^{(k+1)}$ from $(k\alpha+\varepsilon)y$ to $y$. The
concatenation $\zeta^{(k+1)}\circ\zeta^{(k)}\circ\cdots\circ\zeta^{(1)}$
yields a controlled $((1+2\left\Vert y\right\Vert )\varepsilon,\tau)$-chain
from $0$ to $y$.

Since $\varepsilon,\tau>0$ are arbitrary, steps 2 and 3 imply that $y\in
E^{\prime}\cap E$ and hence $E^{\prime}=E$.
\end{proof}

\section{The linear lift and the Poincar\'{e} sphere\label{Section6}}

In this section we lift the affine control semiflow $\Phi$ on $\mathcal{U}%
\times M_{2}$ to a linear control semiflow $\Phi^{1}$ on $\mathcal{U}\times
M_{2}^{1}$ with $M_{2}^{1}:=M_{2}\times\mathbb{R}$ and obtain a discrete
Selgrade decomposition by an application of the generalized Selgrade theorem,
Theorem \ref{Theorem_Selgrade1}. Furthermore, conjugation properties of the
associated semiflows are derived.

The space $M_{2}^{1}$ becomes a Hilbert space with the scalar product%
\[
\left\langle (x,\gamma),(x^{\prime},\gamma^{\prime})\right\rangle
:=\left\langle x.y\right\rangle _{M_{2}}+\gamma\gamma^{\prime}\text{ for
}(x,\gamma),(x^{\prime},\gamma^{\prime})\in M_{2}\times\mathbb{R}.
\]
We embed the linear control system (\ref{D1}) into a bilinear control system
on $M_{2}^{1}$ by introducing an additional state variable $x^{1}$. Consider
for $t\geq0$%
\begin{align}
\dot{x}(t)  &  =A_{0}x(t)+\sum_{i=1}^{p}A_{i}x(t-h_{i})+x^{1}(t)\sum_{i=0}%
^{p}B_{i}u(t-h_{i}),\,u\in\mathcal{U},\nonumber\\
\dot{x}^{1}(t)  &  =0,\label{D3}\\
x(0)  &  =r,\,x(s)=f(s)\text{ for almost all}-h\leq s<0\text{, and }%
x^{1}(0)=\gamma\in\mathbb{R}.\nonumber
\end{align}
Denote the solutions of (\ref{D3}) by $(x(t),x^{1}(t))=\psi^{1}(t,r,f,\gamma
,u)\in\mathbb{R}^{n+1},t\geq0$ solving by (\ref{x_inhom})%
\[
x(t)=e^{A_{0}t}r+\int_{0}^{t}e^{A_{0}(t-s)}\left[  \sum_{i=1}^{p}%
A_{i}x(s-h_{i})+\gamma\sum_{i=0}^{p}B_{i}u(t-h_{i})\right]  ds,\,x^{1}%
(t)=\gamma.
\]
In the state space $M_{2}^{1}$ one obtains the bilinear control system%
\begin{equation}
\dot{y}(t)=Ay(t)+y^{1}(s)B\left(  u(s),u(s-h_{1}),\ldots,u(s-h_{p})\right)
,\dot{y}^{1}(t)=0, \label{linear_lift}%
\end{equation}
with $(y(0),y^{1}(0))=\left(  y_{0},\gamma\right)  \in M_{2}^{1}=M_{2}%
\times\mathbb{R}$ and controls $u\in\mathcal{U}$.

Equivalently, the control semiflow $\Phi$ is lifted to a control semiflow
$\Phi^{1}$ on $\mathcal{U}\times M_{2}^{1}$ defined by%
\begin{align}
\Phi_{t}^{1}(u,y_{0},\gamma)  &  :=(\theta_{t}u,\varphi^{1}(t,y_{0}%
,\gamma,u)),\label{lift1}\\
\varphi^{1}(t,y_{0},\gamma,u)  &  :=\left(  T(t)y_{0}+\int_{0}^{t}%
T(t-s)y^{1}(s)B\left(  u(s),\ldots,u(s-h_{p})\right)  ds,y^{1}(t)\right)
\nonumber\\
&  =\left(  T(t)y_{0}+\gamma\int_{0}^{t}T(t-s)B\left(  u(s),\ldots
,u(s-h_{p})\right)  ds,\gamma\right)  .\nonumber
\end{align}
Observe that for $t\geq0,y_{0}=(r,f)\in M_{2},u\in\mathcal{U}$, and $\gamma=1$
one has
\[
\psi^{1}(t,r,f,1,u)=(\psi(t,r,f,u),1)\text{ and }\varphi^{1}(t,y_{0}%
,1,u)=(\varphi(t,y_{0},u),1).
\]
An application of Theorem \ref{Theorem_Selgrade1}\ to the linear skew product
semiflow $\Phi^{1}$ defined by (\ref{lift1}) yields the following discrete
Selgrade decomposition of $\Phi^{1}$.

\begin{theorem}
\label{Theorem_Selgrade2}For the affine delay control system (\ref{D1}),
assume that $\det A_{p}\not =0$, and consider the associated bilinear delay
control system (\ref{D3}). Then, for the linear control semiflow $\Phi^{1}$ on
$\mathcal{V}^{1}=\mathcal{U}\times M_{2}^{1}$ defined in (\ref{lift1}), there
is an at-most countable sequence $\{\mathcal{A}_{i}^{1}\}_{i=0}^{N^{1}%
},\allowbreak N^{1}\in\{0,1,\ldots\}\cup\{\infty\}$, of subsets of
$\mathbb{P}\mathcal{V}^{1}=\mathcal{U}\times\mathbb{P}M_{2}^{1}$ with
$\mathcal{A}_{0}^{1}=\varnothing,\mathcal{A}_{i}^{1}\subset\mathcal{A}%
_{i+1}^{1}$ for all $1\leq i<N^{1}$, such that, for $1\leq i<N^{1}+1$, the
following properties hold.

(i) The set $\mathcal{A}_{i}^{1}$ is an asymptotically compact attractor for
$\mathbb{P}\Phi^{1}$.

(ii) If $\mathcal{A}^{1}$ is any nonempty asymptotically compact attractor for
$\mathbb{P}\Phi^{1}$, then it follows that $\mathcal{A}^{1}=\mathcal{A}%
_{i}^{1}$ for some $1\leq i<N^{1}+1$.

(iii) For the finite dimensional subbundles $\mathcal{V}_{i}^{1,+}%
=\mathbb{P}^{-1}\mathcal{A}_{i}^{1}$ there are subbundles $\mathcal{V}%
_{i}^{1,-}$ such that $\mathcal{V}^{1}=\mathcal{V}_{i}^{1,+}\oplus
\mathcal{V}_{i}^{1,-}$ is an exponentially separated splitting.

(iv) The subbundles $\mathcal{V}_{i}^{1}:=\mathcal{V}_{i}^{1,+}\cap
\mathcal{V}_{i-1}^{1,-}$ are finite dimensional, invariant subbundles of
$\mathcal{V}^{1}$ such that%
\[
\mathcal{V}_{i}^{1,+}=\mathcal{V}_{1}^{1}\oplus\cdots\oplus\mathcal{V}_{i}%
^{1}.
\]

(v) The sets $\mathcal{M}_{i}^{1}=\mathbb{P}\mathcal{V}_{i}^{1}$ are maximal
chain transitive for the projectivized flow $\mathbb{P}\Phi^{1}$ restricted to
$\mathcal{V}_{i}^{1,+}$.
\end{theorem}

\begin{proof}
We verify the assumptions of Theorem \ref{Theorem_Selgrade1}. The Hilbert
space $M_{2}^{1}$ is separable, and the lifted control semiflow $\Phi^{1}$ on
$\mathcal{U}\times M_{2}^{1}$ is linear. This is seen as in the finite
dimensional case; cf. Colonius and Santana \cite{ColS24}. Theorem
\ref{Theorem_continuous} implies that it is continuous and Hypotheses (H1) is
clear by definition.

By Proposition \ref{Proposition_injective2}, invertibility of the matrix
$A_{p}$ is equivalent to injectivity of the maps $\Phi_{t}(u,\cdot)$. The
definition in (\ref{lift1}) shows that this is also equivalent to injectivity
of the maps $\Phi_{t}^{1}(u,\cdot)$ since $\Phi_{t}^{1}(u,y_{0},\gamma
)=\Phi_{t}^{1}(u,y_{0}^{\prime},\gamma^{\prime})$ holds if and only if
$\gamma=\gamma^{\prime}$ and $\Phi_{t}(u,y_{0})=\Phi_{t}(u,y_{0}^{\prime})$.
Thus hypothesis (H2) holds.

We claim that for fixed $t\geq0$ the map $u\mapsto\Phi_{t}^{1}(u,\cdot)$ is
continuous in the operator norm, thus (H3) holds. For the proof of the claim,
consider for $u^{k}\rightarrow u^{0}$ in $\mathcal{U}$ the difference in the
operator norm on $M_{2}^{1}$. By the definition in (\ref{lift1}) we obtain%
\begin{align*}
&  \left\Vert \Phi_{t}^{1}(u^{k},y_{0},\gamma)-\Phi_{t}^{1}(u^{0},y_{0}%
,\gamma)\right\Vert \\
&  =\sup\left\{  \left\Vert \varphi^{1}(t,y_{0},\gamma,u^{k})-\varphi
^{1}(t,y_{0},\gamma,u^{0})\right\Vert \left\vert \left\Vert (y_{0}%
,\gamma)\right\Vert \leq1\right.  \right\} \\
&  \leq\sup_{\left\vert \gamma\right\vert \leq1}\left\vert \gamma\right\vert
\left\Vert \int_{0}^{t}T(t-s)[B\left(  u^{k}(s),\ldots,u^{k}(s-h_{p})\right)
-B\left(  u^{0}(s),\ldots,u^{0}(s-h_{p})\right)  ]ds\right\Vert \\
&  =\left\Vert \varphi(t,0,u^{k})-\varphi(t,0,u^{0})\right\Vert .
\end{align*}
By Theorem \ref{Theorem_continuous}(i) the right hand side converges to $0$
for $k\rightarrow\infty$, and hence the claim follows.

In order to verify hypotheses (H4), let $(t_{k},u^{k})\rightarrow(t_{0}%
,u^{0})$ and consider $(y_{0},\gamma)\in M_{2}^{1}$. Then, as for (H3),
Theorem \ref{Theorem_continuous}(i) implies $\varphi^{1}(t_{k},u^{k}%
,y_{0},\gamma)\rightarrow\varphi^{1}(t_{0},u^{0},y_{0},\gamma)$ in $M_{2}^{1}%
$. Hence the mapping associating to $(t,u)\in\lbrack0,\infty)\times
\mathcal{U}$ the operator $\Phi_{t}^{1}(u,\cdot)$ on $M_{2}^{1}$ is continuous
in the strong operator topology showing (H4).

Thus all assumptions of Theorem \ref{Theorem_Selgrade1} are verified and the
assertions follow.
\end{proof}

Next we will analyze the Selgrade bundles $\mathcal{V}_{i}^{1}$ in more
detail. Define subsets of $M_{2}^{1}=M_{2}\times\mathbb{R}$, of the unit
sphere $\mathbb{S}M_{2}^{1}=\{(y,\gamma)\in M_{2}^{1}\left\vert \left\Vert
(y,\gamma)\right\Vert =1\right.  \}$, and of the projective space
$\mathbb{P}M_{2}^{1}$ by%
\begin{align*}
M_{2}^{1,0}  &  =M_{2}\times\left\{  0\right\}  ,M_{2}^{1,1}=M_{2}%
\times\left(  \mathbb{R}\setminus\{0\}\right)  ,\\
\mathbb{S}^{+}M_{2}^{1}  &  :=\left\{  (y,\gamma)\in\mathbb{S}M_{2}%
^{1}\left\vert y\in M_{2},\gamma>0\right.  \right\}  ,\mathbb{S}^{0}M_{2}%
^{1}=\{(y,0)\in\mathbb{S}M_{2}^{1}\left\vert y\in M_{2}\right.  \},\\
\mathbb{P}M_{2}^{1,0}  &  =\{\mathbb{P}(y,0)\in\mathbb{P}M_{2}^{1}\left\vert
y\in M_{2}\right.  \},\,\,\mathbb{P}M_{2}^{1,1}=\{\mathbb{P}(y,\gamma
)\in\mathbb{P}M_{2}^{1}\left\vert y\in M_{2},\gamma\not =0\right.  \},
\end{align*}
respectively. One easily sees that the projective space $\mathbb{P}M_{2}%
^{1}=\overline{\mathbb{P}M_{2}^{1,1}}$ is the disjoint union of the closed
subset $\mathbb{P}M_{2}^{1,0}$ and the open subset $\mathbb{P}M_{2}^{1,1}$.
Note that $\mathbb{P}M_{2}^{1,1}$ can be identified with the northern
hemisphere $\mathbb{S}^{+}M_{2}^{1}$ of the sphere $\mathbb{S}M_{2}^{1}$, the
set $\mathbb{S}^{0}M_{2}^{1}$ is the equator, and $\mathbb{P}M_{2}^{1,0}$ is
its image in $\mathbb{P}M_{2}^{1}$.

\begin{definition}
The Poincar\'{e} sphere bundle and the projective Poincar\'{e} bundle are
$\mathcal{U}\times\mathbb{S}M_{2}^{1}$ and $\mathcal{U}\times\mathbb{P}%
M_{2}^{1}$, respectively. The equatorial bundle and the projective equatorial
bundle are $\mathcal{U}\times M_{2}^{1,0}$ and $\mathcal{U}\times
\mathbb{P}M_{2}^{1,0}$, respectively.
\end{definition}

\begin{remark}
\label{Remark_sphere}The construction above is a modification of a classical
construction for polynomial ordinary differential equations going back to
Poincar\'{e}. We consider affine equations and subjoin the additional scalar
state variable $x^{1}$ in front of the inhomogeneous term, while
Poincar\'{e}'s construction just adds $x^{1}$; cf. Perko \cite{Perko}, Cima
and Llibre \cite{CimL90}. A consequence (see Proposition \ref{Proposition_e1})
is that the induced equation on the equatorial bundle is determined by the
linear part and the inhomogeneous term vanishes.
\end{remark}

A conjugacy of two semiflows $\psi$ and $\psi^{\prime}$ on metric spaces $X$
and $X^{\prime}$, respectively, is a homeomorphism $h:X\rightarrow X^{\prime}$
satisfying for all $x\in X$,%
\[
h(\psi(t,x))=\psi^{\prime}(t,h(x))\text{ for }t\geq0.
\]
Recall that the homogeneous part $\Phi^{0}$ of the semiflow $\Phi$ on
$\mathcal{U}\times\mathbb{P}M_{2}$ (cf. (\ref{phi_0})) induces a projectivized
semiflow $\mathbb{P}\Phi_{t}^{0}=(\theta_{t}u,\mathbb{P}T(t)),t\geq0$, on
$\mathcal{U}\times\mathbb{P}M_{2}$.

\begin{proposition}
\label{Proposition_e1}(i) The map%
\[
h^{0}:\mathcal{U}\times M_{2}\rightarrow\mathcal{U}\times M_{2}^{1,0}%
,h^{0}(u,y):=(u,(y,0)),
\]
and its inverse are uniformly continuous and $h^{0}$ conjugates the semiflow
$\Phi^{0}$ on $\mathcal{U}\times M_{2}$ and the semiflow $\Phi^{1}$ restricted
to the equatorial bundle $\mathcal{U}\times M_{2}^{1,0}$.

(ii) The projective map%
\[
\mathbb{P}h^{0}:\mathcal{U}\times\mathbb{P}M_{2}\rightarrow\mathcal{U}%
\times\mathbb{P}M_{2}^{1,0},\mathbb{P}h^{0}(u,\mathbb{P}y):=(u,\mathbb{P}%
(y,0)),
\]
and its inverse are uniformly continuous and $\mathbb{P}h^{0}$ conjugates the
flow $\mathbb{P}\Phi^{0}$ on $\mathcal{U}\times\mathbb{P}M_{2}$ and the flow
$\mathbb{P}\Phi^{1}$ restricted to $\mathcal{U}\times\mathbb{P}M_{2}^{1,0}$.

(iii) For $j\leq i$, the maximal invariant chain transitive sets
$\mathcal{M}_{j}^{0}$ of $\mathbb{P}\Phi^{0}$ restricted to $\mathcal{V}%
_{i}^{0,+}$ are mapped onto the maximal invariant chain transitive sets
$\mathcal{M}_{j}^{1}=\mathbb{P}h^{0}(\mathcal{M}_{j}^{0})$ of $\mathbb{P}%
\Phi^{1}$ restricted to $\mathbb{P}h^{0}(\mathbb{P}\mathcal{V}_{i}^{0,+})$,
and their order is preserved.

(iv) The sets $\mathbb{P}^{-1}(\mathbb{P}h^{0}(\mathcal{M}_{j}^{0}%
))=\mathcal{V}_{j}^{1}$ are finite dimensional subbundles of the subbundle
$\mathbb{P}^{-1}\left(  \mathbb{P}h^{0}(\mathcal{V}_{i}^{0,+})\right)
\subset\mathcal{U}\times M_{2}^{1,0}$.
\end{proposition}

\begin{proof}
(i) The semiflow $\Phi^{1}$ restricted to $\mathcal{U}\times M_{2}^{1,0}$ and
the semiflow $\Phi^{0}$ on $\mathcal{U}\times M_{2}$ satisfy%
\[
\Phi_{t}^{1}(h^{0}(u,y))=\Phi_{t}^{1}(u,y,0)=(\theta_{t}u,T(t)y,0)=(\Phi
_{t}^{0}(u,y),0)=h^{0}(\Phi_{t}^{0}(u,y)).
\]
Furthermore, uniform continuity of $h^{0}$ and $(h^{0})^{-1}$ holds since%
\[
d((u,y),(v,z))=d((u,(y,0)),(v,(z,0))).
\]
Assertion (ii) is a consequence of (i). Since the map $\mathbb{P}h^{0}$ and
its inverse are uniformly continuous they preserve chain transitivity. By
Proposition \ref{Proposition_Selgrade_hom} the subbundles $\mathcal{V}_{i}%
^{0}$ and the subsets $\mathbb{P}\mathcal{V}_{i}^{0}$ are linearly ordered
(cf. Remark \ref{Remark_order}) and the sets $\mathcal{M}_{i}^{0}%
=\mathbb{P}\mathcal{V}_{i}^{0}$ are maximal chain transitive sets for
$\mathbb{P}\Phi^{0}$ restricted to $\mathbb{P}\mathcal{V}_{i}^{0,+}$. This
also implies the assertion on the maximal chain transitive sets and the
associated subbundles in (iii) and assertion (iv).
\end{proof}

Next we turn to the induced semiflow $\mathbb{P}\Phi^{1}$ restricted to
$\mathcal{U}\times\mathbb{P}M_{2}^{1,1}$ and note the following lemma.

\begin{lemma}
\label{Lemma_epsilon}Define the map%
\begin{equation}
h^{1}:\mathcal{U}\times M_{2}\rightarrow\mathcal{U}\times\mathbb{P}M_{2}%
^{1,1},h^{1}(u,y):=(u,\mathbb{P}\left(  y,1\right)  ),(u,y)\in\mathcal{U}%
\times M_{2}. \label{h1}%
\end{equation}
For every $\varepsilon,\tau>0$ any $(\varepsilon,\tau)$-chain in
$\mathcal{U}\times M_{2}$ is mapped by $h^{1}$ onto a $(2\varepsilon,T)$-chain
in $\mathcal{U}\times\mathbb{P}M_{2}^{1,1}$.
\end{lemma}

\begin{proof}
It suffices to show that $d((u,y),(u^{\prime},y^{\prime}))<\varepsilon$
implies $d(h^{1}(u,y),h^{1}(u^{\prime},y^{\prime}))<2\varepsilon$ in
$\mathcal{U}\times\mathbb{P}M_{2}^{1}$ and this follows by the following
estimates of the distances on $\mathbb{P}M_{2}^{1}$. According to the
definition of the metric in (\ref{metric_P}) on projective space it suffices
to estimate
\[
\left\Vert \frac{(y,1)}{\left\Vert (y,1)\right\Vert }-\frac{(y^{\prime}%
,1)}{\left\Vert (y^{\prime},1)\right\Vert }\right\Vert =\left(  \frac
{y}{\left\Vert (y,1)\right\Vert }-\frac{y^{\prime}}{\left\Vert (y^{\prime
},1)\right\Vert },\frac{1}{\left\Vert (y,1)\right\Vert }-\frac{1}{\left\Vert
(y^{\prime},1)\right\Vert }\right)  .
\]
Note that $\left\Vert y\right\Vert -\left\Vert y^{\prime}\right\Vert
\leq\left\Vert y-y^{\prime}\right\Vert <\varepsilon$ and $\left\Vert
(y,1)\right\Vert -\left\Vert (y^{\prime},1)\right\Vert \leq\left\Vert
(y-y^{\prime},0)\right\Vert =\left\Vert y-y^{\prime}\right\Vert <\varepsilon$.
Hence we find $\delta(\varepsilon)$ with $\left\vert \delta(\varepsilon
)\right\vert <\varepsilon$ such that $\left\Vert (y^{\prime},1)\right\Vert
=\left\Vert (y,1)\right\Vert +\delta(\varepsilon)$. The last component
satisfies%
\[
\frac{1}{\left\Vert (y,1)\right\Vert }-\frac{1}{\left\Vert (y^{\prime
},1)\right\Vert }=\frac{\left\Vert (y^{\prime},1)\right\Vert -\left\Vert
(y,1)\right\Vert }{\left\Vert (y,1)\right\Vert \left\Vert (y^{\prime
},1)\right\Vert }<\varepsilon
\]
and%
\begin{align*}
\left\Vert \left\Vert (y^{\prime},1)\right\Vert y-\left\Vert (y,1)\right\Vert
y^{\prime}\right\Vert  &  =\left\Vert \left[  \left\Vert (y,1)\right\Vert
+\delta(\varepsilon)\right]  y-\left\Vert (y,1)\right\Vert y^{\prime
}\right\Vert \\
&  \leq\left\Vert (y,1)\right\Vert \left\Vert y-y^{\prime}\right\Vert
+\delta(\varepsilon)\left\Vert y\right\Vert <\left\Vert (y,1)\right\Vert
\varepsilon+\delta(\varepsilon)\left\Vert y\right\Vert .
\end{align*}
Hence, the other components satisfy%
\begin{align*}
\left\Vert \frac{y}{\left\Vert (y,1)\right\Vert }-\frac{y^{\prime}}{\left\Vert
(y^{\prime},1)\right\Vert }\right\Vert  &  =\frac{\left\Vert ~\left\Vert
(y^{\prime},1)\right\Vert y-\left\Vert (y,1)\right\Vert y^{\prime}\right\Vert
}{\left\Vert (y,1)\right\Vert ~\left\Vert (y^{\prime},1)\right\Vert }\leq
\frac{\left(  \left\Vert (y,1)\right\Vert \varepsilon+\delta(\varepsilon
)\left\Vert y\right\Vert \right)  }{\left\Vert (y,1)\right\Vert ~\left\Vert
(y^{\prime},1)\right\Vert }\\
&  <\varepsilon+\left\vert \delta(\varepsilon)\right\vert <2\varepsilon.
\end{align*}
This implies the desired estimate.
\end{proof}

The following proposition shows the semiflow $\Phi$ on $\mathcal{U}\times
M_{2}$ is conjugate to the semiflow $\mathbb{P}\Phi^{1}$ restricted to
$\mathcal{U}\times\mathbb{P}M_{2}^{1,1}$. Furthermore, a subset of
$\mathcal{U}\times M_{2}$ is unbounded if and only if the closure of its image
in $\mathcal{U}\times\mathbb{P}M_{2}^{1}$ intersects the projective equatorial bundle.

\begin{proposition}
\label{Proposition_e2}(i) The map $h^{1}$ defined in (\ref{h1}) is a conjugacy
of the semiflows $\Phi$ on $\mathcal{U}\times M_{2}$ and $\mathbb{P}\Phi^{1}$
restricted to $\mathcal{U}\times\mathbb{P}M_{2}^{1,1}$,%
\[
h^{1}(\Phi_{t}(u,y))=\mathbb{P}\Phi_{t}^{1}(u,y,1)\text{ for }t\geq0.
\]

(ii) For a subset $C\subset\mathcal{U}\times M_{2}$ the set $\{y\in
M_{2}\left\vert (u,y)\in C\text{ for some }u\in\mathcal{U}\right.  \}$ is
bounded if and only if $\overline{h^{1}(C)}\cap(\mathcal{U}\times
\mathbb{P}M_{2}^{1,0})=\varnothing$.
\end{proposition}

\begin{proof}
(i) The proof of Lemma \ref{Lemma_epsilon} shows that $h^{1}$ is continuous.
The first component of $h^{1}$ is the identity on $\mathcal{U}$. Concerning
the second component, suppose that $\mathbb{P}(y,1)=\mathbb{P}(y^{\prime},1)$,
i.e., $(y,1)=\lambda(y^{\prime},1)$ for some $\lambda\not =0$. This implies
$\lambda=1$ and hence $y=y^{\prime}$. Thus $h^{1}$ is injective, and it
certainly is surjective. It remains to show that the inverse of $h^{1}$ is
continuous. Suppose that
\[
d_{\mathbb{P}}(\mathbb{P}(y_{k},1),\mathbb{P}(y,1))=\min\left\{  \left\Vert
\frac{(y_{k},1)}{\left\Vert (y_{k},1)\right\Vert }-\frac{(y,1)}{\left\Vert
(y,1)\right\Vert }\right\Vert ,\left\Vert \frac{(y_{k},1)}{\left\Vert
(y_{k},1)\right\Vert }+\frac{(y,1)}{\left\Vert (y,1)\right\Vert }\right\Vert
\right\}  \rightarrow0.
\]
The second terms cannot converge to $0$, since the last component is greater
than or equal to $\frac{1}{\left\Vert (y,1)\right\Vert }$. Hence we know that%
\[
\left\Vert \frac{(y_{k},1)}{\left\Vert (y_{k},1)\right\Vert }-\frac
{(y,1)}{\left\Vert (y,1)\right\Vert }\right\Vert =\left\Vert \left(
\frac{y_{k}}{\left\Vert (y_{k},1)\right\Vert }-\frac{y}{\left\Vert
(y,1)\right\Vert },\frac{1}{\left\Vert (y_{k},1)\right\Vert }-\frac
{1}{\left\Vert (y,1)\right\Vert }\right)  \right\Vert \rightarrow0.
\]
The last components converge to $0$ implying $\left\Vert (y_{k},1)\right\Vert
\rightarrow\left\Vert (y,1)\right\Vert $. Since also the other components
converge to $0$ we conclude that $\left\Vert y_{k}-y\right\Vert \rightarrow0$.
This shows that $h^{1}$ is a homeomorphism. The conjugacy property follows by%
\[
h^{1}(\Phi_{t}(u,y))=(\theta_{t}u,\mathbb{P}\left(  \varphi(t,y,u),1\right)
)=(\theta_{t}u,\mathbb{P}\varphi^{1}(t,y,u,1))=\mathbb{P}\Phi_{t}%
^{1}(u,y,1),t\geq0.
\]

(ii) Consider a sequence $(u^{k},y_{k}),k\in\mathbb{N}$, in $C$. For the
images $h^{1}(u^{k},y_{k})=(u^{k},\mathbb{P}(y_{k},1))$, the points
$\mathbb{P}(y_{k},1)$ are determined by%
\[
\pm\left(  \frac{y_{k}}{\left\Vert (y_{k},1)\right\Vert ~},\frac{1}{\left\Vert
(y_{k},1)\right\Vert }\right)  .
\]
Then it follows that $\left\Vert y_{k}\right\Vert \rightarrow\infty$ for
$k\rightarrow\infty$, is equivalent to the property that the distances of
$(u^{k},\mathbb{P}(y_{k},1))$ to $\mathcal{U}\times\mathbb{P}M_{2}^{1,0}$
converge to $0$.
\end{proof}

\section{Hyperbolic semiflows\label{Section7}}

This section analyzes hyperbolic control semiflows for delay systems. Again we
assume throughout that $\det A_{p}\not =0$.

Recall that a hyperbolic homogeneous delay equation yields by Theorem
\ref{Theorem_hyperbolic_delay} a spectral decomposition of $M_{2}$ into a
finite dimensional subspace $V^{+}$ and a stable subspace $V^{-}$. Since the
homogeneous part $\Phi^{0}$ of the associated control semiflow is the product
semiflow $\Phi_{t}^{0}(u,y)=(\theta_{t}u,T(t)y)$, this flow is also hyperbolic
with the following decomposition into closed subbundles $\mathcal{U}\times
M_{2}=\mathcal{V}^{-}\oplus\mathcal{V}^{+}$, $\mathcal{V}^{-}:=\mathcal{U}%
\times V^{-}$ and $\mathcal{V}^{+}:=\mathcal{U}\times V^{+}$. There are
constants $\alpha,K>0$ such that%
\begin{align*}
\left\Vert \Phi_{t}^{0}(u,y^{-})\right\Vert  &  =\left\Vert T(t)y^{-}%
\right\Vert \leq Ke^{-\alpha t}\left\Vert y^{-}\right\Vert \text{ for }%
t\geq0\text{ and }(u,y^{-})\in\mathcal{V}^{-},\\
\left\Vert \Phi_{t}^{0}(u,y^{+})\right\Vert  &  =\left\Vert T(t)y^{+}%
\right\Vert \leq Ke^{\alpha t}\left\Vert y^{+}\right\Vert \text{ for }%
t\leq0\text{ and }(u,y^{+})\in\mathcal{V}^{+}.
\end{align*}
Since $\dim V^{+}<\infty$ the solution map $T(t)$ is an isomorphism on the
invariant subspace $V^{+}$ and hence, for every $y\in V^{+}$, there exists an
entire solution $\varphi^{0}(t,y),t\in\mathbb{R}$.

Next we consider the inhomogeneous equation (\ref{D2}). Let $\pi^{\pm}%
:M_{2}\rightarrow V^{\pm}$ be the associated projections and define $t\geq0$%
\[
\varphi^{\pm}(t,u,y^{\pm}):=T(t)y^{\pm}+\int_{0}^{t}T(t-s)\pi^{\pm
}B(u(s),\ldots,u(s-h_{p})ds\text{ for }(u,y^{\pm})\in\mathcal{U}\times V^{\pm
},
\]
and define associated affine semiflows on $\mathcal{V}^{\pm}:=\mathcal{U}%
\times V^{\pm}$ by%
\[
\Phi_{t}^{\pm}(u,y):=\left(  \theta_{t}u,\varphi^{\pm}(t,u,y)\right)  \text{
for }t\geq0\text{ and }(u,x^{\pm})\in\mathcal{V}^{\pm}.
\]
Our next goal is to prove that for every $u\in\mathcal{U}$ there exists a
unique bounded solution of $\Phi$. We start with the stable part.

\begin{lemma}
\label{Lemma7.2}Assume that the linear part $\Phi^{0}$ of the affine semiflow
$\Phi$ is hyperbolic. Then for every $u\in\mathcal{U}$ there exists a unique
entire bounded solution $(\theta_{t}u,e^{-}(u,t)),\allowbreak t\in\mathbb{R}$
of the affine semiflow $\Phi^{-}$. It satisfies $e^{-}(\theta_{t}%
u,0)=e^{-}(u,t)$ for $t\in\mathbb{R}$, and the map $e^{-}:\mathcal{U}%
\times\mathbb{R}\rightarrow M_{2}$ is continuous.
\end{lemma}

\begin{proof}
First we show that the linear semiflow $T(\cdot)$ restricted to $V^{-}$ has
only the trivial entire bounded solution. Any entire bounded solution
$\varphi^{0}(t,y^{-}),t\in\mathbb{R}$ satisfies, for $t\geq0$%
\begin{align*}
\left\Vert y^{-}\right\Vert  &  =\left\Vert \varphi^{0}(0,y^{-})\right\Vert
=\left\Vert \varphi^{0}(t,\varphi^{0}(-t,y^{-})\right\Vert \leq Ke^{-\alpha
t}\left\Vert \varphi^{0}(-t,y^{-})\right\Vert \\
&  \leq Ke^{-\alpha t}\sup\nolimits_{s\leq0}\left\Vert \varphi^{0}%
(s,y^{-})\right\Vert .
\end{align*}
The right hand side converges to $0$ for $t\rightarrow\infty$, hence $y^{-}%
=0$. It immediately follows that there is at most a single entire bounded
solution for $\Phi^{-}$ since the difference of two bounded entire solutions
in $V^{-}$ is a bounded entire solution in $V^{-}$ of the linear semiflow. We
claim that%
\[
e^{-}(u,t):=\int_{-\infty}^{t}T(t-s)\pi^{-}B(u(s),\ldots,u(s-h_{p}%
))ds,t\in\mathbb{R},
\]
is the desired solution. The integral exists since $t-s\geq0$ for
$s\in(-\infty,t)$ and for all $u\in\mathcal{U}$ and $s\leq t$%
\begin{align*}
&  \left\Vert T(t-s)\pi^{-}B(u(s),\ldots,u(s-h_{p}))\right\Vert \leq
Ke^{-\alpha(t-s)}\left\Vert \pi^{-}B(u(s),\ldots,u(s-h_{p}))\right\Vert \\
&  \leq Ke^{-\alpha(t-s)}\left\Vert \pi^{-}\right\Vert (p+1)\max
_{i=0,\ldots,p}\left\Vert B_{i}\right\Vert \max_{u\in\Omega}\left\Vert
u\right\Vert .
\end{align*}
This is a solution since for $t_{0}\in\mathbb{R}$ and $t\geq t_{0}\geq0$ it
satisfies formula (\ref{mild}) for the initial value $e^{-}(u,t_{0})$:%
\begin{align*}
e^{-}(u,t)  &  =\int_{-\infty}^{t}T(t-s)\pi^{-}B(u(s),\ldots,u(s-h_{p}))ds\\
&  =T(t-t_{0})\int_{-\infty}^{t_{0}}T(t_{0}-s)\pi^{-}B(u(s),\ldots
,u(s-h_{p}))ds\\
&  \qquad+\int_{t_{0}}^{t}T(t-s)\pi^{-}B(u(s),\ldots,u(s-h_{p}))ds\\
&  =T(t-t_{0})e^{-}(u,t_{0}))+\int_{t_{0}}^{t}T(t-s)\pi^{-}B(u(s),\ldots
,u(s-h_{p}))ds.
\end{align*}
Note that for $t\in\mathbb{R}$%
\begin{align*}
e^{-}(\theta_{t}u,0)  &  =\int_{-\infty}^{0}T(-s)\pi^{-}B(u(t+s),\ldots
,u(t+s-h_{p}))ds\\
&  =\int_{-\infty}^{t}T(t-s)\pi^{-}B(u(s),\ldots,u(s-h_{p}))ds=e^{-}(u,t).
\end{align*}
In order to prove continuity let $u,u^{0}\in\mathcal{U}$ and $t,t_{0}%
\in\mathbb{R}$. Then%
\begin{align*}
&  \left\Vert e^{-}(u,t)-e^{-}(u^{0},t_{0})\right\Vert \\
&  =\left\Vert \int_{-\infty}^{t}T(t-s)\pi^{-}B(u(s),\ldots,u(s-h_{p}%
))ds\right. \\
&  \qquad\left.  -\int_{-\infty}^{t_{0}}T(t_{0}-s)\pi^{-}B(u^{0}%
(s),\ldots,u^{0}(s-h_{p}))ds\right\Vert \\
&  \leq\left\Vert \int_{\mathbb{R}}\left[  \chi_{(-\infty,t]}(s)T(t-s)-\chi
_{(-\infty,t_{0}]}(s)T(t_{0}-s)\right]  \pi^{-}B(u(s),\ldots,u(s-h_{p}%
))ds\right\Vert \\
&  \qquad+\left\Vert \int_{-\infty}^{t_{0}}T(t_{0}-s)\pi^{-}B(u(s)-u^{0}%
(s),\ldots,u(s-h_{p})-u^{0}(s-h_{p}))ds\right\Vert .
\end{align*}
For $(t,u)\rightarrow(t_{0},u^{0})$ the first summand converges to $0$ by
strong continuity of $T(\cdot)$ and Lebesgue's theorem. The integrand in the
second summand is%
\[
\varphi^{-}(t_{0}-s,0,u)-\varphi^{-}(t_{0}-s,0,u^{0}).
\]
For $u\rightarrow u^{0}$ in $\mathcal{U}$, Theorem \ref{Theorem_continuous}(i)
implies that this converges to $0$, for every $s\in(-\infty,t_{0}]$. Again
Lebesgue's theorem implies that the integral converges to $0$.
\end{proof}

An analogous result holds for the unstable part.

\begin{lemma}
\label{Lemma7.3}Assume that the linear part $\Phi^{0}$ of the affine semiflow
$\Phi$ is hyperbolic. Then for every $u\in\mathcal{U}$ there exists a unique
entire bounded solution $(\theta_{t}u,e^{+}(u,t)),\allowbreak t\in\mathbb{R}$
of the affine semiflow $\Phi$. It satisfies $e^{+}(\theta_{t}u,0)=e^{+}(u,t)$
for $t\in\mathbb{R}$, and the map $e^{+}:\mathcal{U}\times\mathbb{R}%
\rightarrow M_{2}$ is continuous.
\end{lemma}

\begin{proof}
Let $T(t)y^{+}=\varphi^{0}(t,y^{+}),t\in\mathbb{R}$, be a bounded solution for
$T(t)$ restricted to $V^{+}$. Then it follows, for $t\geq0$,%
\[
\left\Vert \varphi^{0}(t,y^{+})\right\Vert \geq Ke^{\alpha t}\left\Vert
y^{+}\right\Vert \rightarrow\infty\text{ for }k\rightarrow\infty.
\]
This implies $y^{+}=0$. As above there is at most a single entire bounded
solution for $\Phi^{+}$ since the difference of two bounded entire solutions
is a bounded entire solution for the homogeneous semiflow.

Next we show that the entire bounded solution is given by%
\[
e^{+}(u,t)=\int_{-\infty}^{t}T(t+s)\pi^{+}B(u(s),\ldots,u(s-h_{p}%
))ds,t\in\mathbb{R}.
\]
Observe that the integrand is well defined, since $\pi^{+}$ is a map onto the
finite dimensional subspace $V^{+}$ and $T(t+s)$ is an isomorphism on $V^{+}$.
Existence of the integral follows from%
\[
\left\Vert T(t+s)y\right\Vert \leq K^{-1}e^{\alpha s}\left\Vert
T(t)y\right\Vert \text{ for }s\leq0\text{ and }y\in V^{+}.
\]
The other assertions follow as in the proof of Lemma \ref{Lemma7.2}.
\end{proof}

A combination of the two previous lemmas establishes the desired unique
existence of entire bounded solutions and shows that the affine semiflow is
conjugate to its homogeneous part; cf. Colonius and Santana \cite[Corollary 1
and Theorem 2.5]{ColSan11} for an analogous result in finite dimensions.

\begin{proposition}
\label{Proposition7.3}Suppose that the linear part $\Phi^{0}$ of the affine
semiflow $\Phi$ is hyperbolic.

(i) Then, for every $u\in\mathcal{U}$, there is a unique bounded entire
solution given by $(\theta_{t}u,e(u,t)),\allowbreak t\in\mathbb{R}$ for the
affine semiflow $\Phi$, the map $e:\mathcal{U}\times\mathbb{R}\rightarrow
M_{2}$ is continuous, and $e(\theta_{t}u,0)=e(u,t)$ for $t\in\mathbb{R}$.

(ii) The affine semiflow $\Phi$ and its linear part are conjugate by the
homeomorphism%
\begin{equation}
H:\mathcal{U}\times M_{2}\rightarrow\mathcal{U}\times M_{2}%
:H(u,y):=(u,y-e(u,0))\text{ for }(u,y)\in\mathcal{U}\times M_{2}. \label{H}%
\end{equation}

\end{proposition}

\begin{proof}
(i) Lemmas \ref{Lemma7.2} and \ref{Lemma7.3} imply the existence of unique
bounded entire solutions $(\theta_{t}u,e^{\pm}(u,t)),\allowbreak
t\in\mathbb{R}$. This yields the bounded entire solution for $\Phi$%
\[
(\theta_{t}u,e(u,t))=(\theta_{t}u,e^{+}(u,t)+e^{-}(u,t))=(\theta_{t}%
u,e^{+}(u,t))+(\theta_{t}u,e^{-}(u,t)),~t\in\mathbb{R}.
\]
Since any bounded entire solution for $\Phi$ induces bounded entire solutions
in $\mathcal{U\times}V^{\pm}$, uniqueness follows. Furthermore, the map
$\mathcal{U}\times\mathbb{R}\rightarrow\mathcal{U}\times M_{2}:(u,t)\mapsto
(u,e(u,t))$ is continuous.

(ii) The map $H$ is continuous and bijective with continuous inverse
\[
H^{-1}(u,y):=(u,y+e(u,0))\text{ for }(u,y)\in\mathcal{U}\times M_{2}.
\]
The conjugation property follows from%
\begin{align*}
H(\Phi_{t}(u,y))  &  =H(\theta_{t}u,\varphi(t,u,y))=(\theta_{t}u,\varphi
(t,u,y)-e(\theta_{t}u,0))\\
&  =(\theta_{t}u,\varphi(t,u,y)-e(u,t))=(\theta_{t}u,\varphi(t,u,y)-\varphi
(t,u,e(u,0)))\\
&  =(\theta_{t}u,\varphi^{0}(t,y-e(u,0))=\Phi_{t}^{0}(H(u,y)).
\end{align*}

\end{proof}

The following lemma shows that the chain recurrent set of uniformly hyperbolic
linear systems is trivial. Antunez, Mantovani, and Var\~{a}o \cite[Corollary
2.11]{AntMV22} prove an analogous result for hyperbolic linear operators on
Banach spaces.

\begin{lemma}
\label{Lemma_trivial}Suppose that $\Phi^{0}$ is hyperbolic with decomposition
$\mathcal{V}=\mathcal{U}\times M_{2}=\mathcal{V}^{+}\oplus\mathcal{V}^{-}$.
Then the chain recurrent set of $\Phi^{0}$ equals $\mathcal{U}\times
\{0_{M_{2}}\}$.
\end{lemma}

\begin{proof}
It is clear that $\mathcal{U}\times\{0_{M_{2}}\}$ is contained in the chain
recurrent set.

(i) First we show that the chain recurrent set of $\mathcal{V}^{-}$ equals
$\mathcal{U}\times\{0_{M_{2}}\}$. Suppose, by way of contradiction, that
$(u,y)\in\mathcal{V}^{-}$ with $y\not =0$ is chain recurrent and consider for
$\varepsilon\in(0,1),\tau>0$ an $(\varepsilon,\tau)$-chain from $(u,y)$ to
$(u,y)$ given by%
\[
\tau_{0},\ldots,\tau_{q-1}\geq\tau\,\text{\ and }d(\Phi_{\tau_{i}}^{0}%
(u^{i},y_{i}),(u^{i+1},y_{i+1}))<\varepsilon\text{ for }i=0,\ldots,q-1.
\]
Let $\tau>0$ such that $\beta:=Ke^{-\alpha\tau}<1$. Then $\left\Vert
y_{i+1}-\varphi(\tau_{i},u^{i},y_{i})\right\Vert <\varepsilon$ implies%
\begin{align*}
\left\Vert y\right\Vert  &  =\left\Vert y_{q}\right\Vert \leq\left\Vert
\varphi(\tau_{q-1},u^{q-1},y_{q-1})\right\Vert +\varepsilon\leq\beta\left\Vert
y_{q-1}\right\Vert +\varepsilon\\
&  \leq\beta^{2}\left\Vert \varphi(\tau_{q-2},u^{q-2},y_{q-2})\right\Vert
+\beta\varepsilon+\varepsilon\\
&  \leq\beta^{q}\left\Vert y\right\Vert +\beta^{q-2}\varepsilon^{q-2}%
+\cdots+\beta\varepsilon+\varepsilon\\
&  <\beta^{q}\left\Vert y\right\Vert +(1-\beta\varepsilon)^{-1}-1+\varepsilon.
\end{align*}
Since $(1-\beta\varepsilon)^{-1}-1+\varepsilon\rightarrow0$ for $\varepsilon
\rightarrow0$ we may take $\varepsilon>0$ small enough such that for any
$q\geq1$,%
\[
(1-\beta^{q})\left\Vert y\right\Vert >(1-\beta\varepsilon)^{-1}-1+\varepsilon
,
\]
and hence $\left\Vert y\right\Vert >\beta^{q}\left\Vert y\right\Vert
+(1-\beta\varepsilon)^{-1}-1+\varepsilon$. This contradiction shows that, for
$\varepsilon>0$ small enough and $\tau>0$ large enough, there are no
$(\varepsilon,\tau)$-chains from $(u,y)$ to $(u,y)$.

(ii) Next we show that the chain recurrent set of $\mathcal{V}^{+}$ equals
$\mathcal{U}\times\{0\}$. Suppose that $(u,y)\in\mathcal{V}^{-}$ with
$y\not =0$ is chain recurrent and consider for $\varepsilon\in(0,1),\tau>0$ an
$(\varepsilon,\tau)$-chain from $(u,y)$ to $(u,y)$. Let $\tau>0$ such that
$\beta:=Ke^{\alpha\tau}>1$. Similarly as in (i) we compute%
\begin{align*}
\left\Vert y\right\Vert  &  =\left\Vert y_{q}\right\Vert \geq\left\Vert
y_{q}-\varphi(\tau_{q-1},u^{q-1},y_{q-1})+\varphi(\tau_{q-1},u^{q-1}%
,y_{q-1})\right\Vert \\
&  \geq\left\Vert \varphi(\tau_{q-1},u^{q-1},y_{q-1})\right\Vert -\left\Vert
y_{q}-\varphi(\tau_{q-1},u^{q-1},y_{q-1})\right\Vert \\
&  >\beta\left\Vert y_{q-1}\right\Vert -\varepsilon\geq\beta^{2}\left\Vert
y_{q-2}\right\Vert -\beta\varepsilon-\varepsilon\\
&  \geq\beta^{q}\left\Vert y\right\Vert -(\beta\varepsilon)^{q-2}%
-(\beta\varepsilon)^{q-3}-\cdots-\beta\varepsilon-\varepsilon.
\end{align*}
Let $\varepsilon>0$ be small enough such that $\beta\varepsilon<1$. Then it
follows that%
\[
\left\Vert y\right\Vert >\beta^{q}\left\Vert y\right\Vert -(1-\beta
\varepsilon)^{-1}+1-\varepsilon.
\]
For $\varepsilon>0$ small enough this contradicts $\beta>1$ and hence, for
$\varepsilon>0$ small enough and $\tau>0$ large enough, there is no
$(\varepsilon,\tau)$-chain from $(u,y)$ to $(u,y)$ if $y\not =0$.

(iii) Any $(\varepsilon,\tau)$-chain in $\mathcal{U}\times\mathbb{R}%
^{n}=\mathcal{V}^{+}\oplus\mathcal{V}^{-}$ projects to $(\varepsilon,\tau)$
chains in $\mathcal{V}^{+}$ and $\mathcal{V}^{-}$. Thus (i) and (ii) imply the assertion.
\end{proof}

The following result characterizes the chain recurrent set for hyperbolic
control semiflows.

\begin{theorem}
\label{Theorem_hyperbolic1}Consider the affine control semiflow $\Phi$ on
$\mathcal{U}\times M_{2}$ defined in (\ref{phi}) associated with the delay
system (\ref{D1}). Suppose that $\det A_{p}\not =0$ and that the linear part
$\Phi^{0}$ of $\Phi$ is hyperbolic. Then, for the linear semiflow $\Phi^{0}$,
the entire chain recurrent set is $\mathcal{R}^{\#}(\Phi^{0})=\mathcal{U}%
\times\{0_{M_{2}}\}$ and, for the affine semiflow $\Phi$%
\[
\mathcal{R}^{\#}(\Phi)=H(\mathcal{R}^{\#}(\Phi^{0}))=\{(u,-e(u,0))\left\vert
u\in\mathcal{U}\right.  \},
\]
where $H$ is the homeomorphism defined in (\ref{H}) and $e(u,t)\in M_{2}%
,t\in\mathbb{R}$ is the unique entire bounded solution of (\ref{hom_delay}).
The entire chain recurrent set $\mathcal{R}^{\#}(\Phi)$ is compact, invariant,
and chain transitive.
\end{theorem}

\begin{proof}
Lemma \ref{Lemma_trivial} shows that the entire chain recurrent set of
$\Phi^{0}$ is $\mathcal{R}^{\#}(\Phi^{0})=\mathcal{U}\times\{0_{M_{2}}\}$. By
Proposition \ref{Proposition7.3} the set%
\[
H(\mathcal{R}^{\#}(\Phi^{0}))=H(\mathcal{U}\times\{0_{M_{2}}%
)=\{(u,-e(u,0))\left\vert u\in\mathcal{U}\right.  \}.
\]
is compact using that $\mathcal{U}$ is compact and $e(\cdot,0)$ is continuous.

The set $\{(u,-e(u,0))\left\vert u\in\mathcal{U}\right.  \}$ is invariant
since, by Proposition \ref{Proposition7.3}(i),%
\[
\Phi_{t}(u,-e(u,0))=(\theta_{t}u,-e(u,t))=(\theta_{t}u,-e(\theta_{t}%
u,0)),t\in\mathbb{R}.
\]
The map $H$ is uniformly continuous: In fact, for $\varepsilon>0$ it follows
by compactness of $\mathcal{U}$ and continuity of $e(\cdot,0)$ that there is
$\delta(\varepsilon)\in(0,\varepsilon/2)$ such that $d(u,u^{\prime}%
)<\delta(\varepsilon)$ and $\left\Vert y-y^{\prime}\right\Vert <\delta
(\varepsilon)$ implies
\[
\left\Vert y-e(u,0)-\left(  y^{\prime}-e(u^{\prime},0)\right)  \right\Vert
\leq\left\Vert y-y^{\prime}\right\Vert +\left\Vert e(u,0)-e(u^{\prime
},0)\right\Vert <\delta(\varepsilon)+\varepsilon/2<\varepsilon.
\]
Hence $d(u,y),(u^{\prime},y^{\prime}))<\delta(\varepsilon)$ implies
$d(H(u,y),H(u^{\prime},y^{\prime}))<\varepsilon$. Analogously one proves that
the inverse of $H$ given by $H^{-1}(u,y)=(u,y+e(u,0))$ is uniformly continuous.

Let $\varepsilon,\tau>0$ and consider $H(u,0),H(u^{\prime},0)\in
H(\mathcal{R}^{\#}(\Phi^{0}))$ with $u,u^{\prime}\in\mathcal{U}$. By chain
transitivity of $\mathcal{U}$ there is a $(\delta(\varepsilon),\tau)$-chain in
$\mathcal{U}\times\{0_{M_{2}}\}$ from $(u,0)$ to $(u^{\prime},0)$. Then $H$
maps this chain onto an $(\varepsilon,\tau)$-chain from $H(u,0)$ to
$H(u^{\prime},0)$. Since $\varepsilon,\tau>0$ are arbitrary, this proves that
$H(\mathcal{R}^{\#}(\Phi^{0}))$ is chain transitive and certainly this set is
invariant and consists of points defining entire solutions.

It remains to prove that $H(\mathcal{R}^{\#}(\Phi^{0}))$ is the entire chain
recurrent set of $\Phi$. Let $\varepsilon>0$. By uniform continuity of
$H^{-1}$ there is $\delta^{\prime}(\varepsilon)>0$ such that
$d(u,y),(u^{\prime},y^{\prime}))<\delta^{\prime}(\varepsilon)$ implies
$d(H^{-1}(u,y),H^{-1}(u^{\prime},y^{\prime}))<\varepsilon$. For any chain
recurrent point $(u,y)$ of $\Phi$ and $\tau>0$ there is a $(\delta^{\prime
}(\varepsilon),\tau)$-chain from $(u,y)$ to $(u,y)$. This is mapped by
$H^{-1}$ to an $(\varepsilon,\tau)$-chain of $\Phi$ from $H^{-1}(u,y)$ to
$H^{-1}(u,y)$. This proves that $H^{-1}(u,y)\in\mathcal{R}^{\#}(\Phi^{0})$ and
hence $(u,y)=H(H^{-1}(u,y))\in H(\mathcal{R}^{\#}(\Phi^{0}))$.
\end{proof}

Next we use the linear lift to describe the image of the entire chain
recurrent set.

\begin{theorem}
\label{Theorem_lift}Consider the delay control system (\ref{D1}) and suppose
that $\det A_{p}\not =0$. Assume, for the associated affine control semiflow
$\Phi$ on $\mathcal{U}\times M_{2}$ defined in (\ref{phi}), that the linear
part $\Phi^{0}$ is hyperbolic.

(i) Then the lift $\Phi^{1}$ on $\mathcal{U}\times M_{2}^{1}$ defined in
(\ref{lift1}) possesses an invariant one dimensional subbundle $\mathcal{V}%
_{c}^{1}$ of $\mathcal{U}\times M_{2}^{1}$ defined by%
\begin{equation}
\mathcal{V}_{c}^{1}=\{(u,-re(u,0),r)\in\mathcal{U}\times M_{2}\times
\mathbb{R}\left\vert u\in\mathcal{U},r\in\mathbb{R}\right.  \}. \label{Vstar2}%
\end{equation}

(ii) The projection $\mathcal{M}_{c}^{1}=\mathbb{P}\mathcal{V}_{c}^{1}$ to
$\mathcal{U}\times\mathbb{P}M_{2}^{1}$ is a compact subset of $\mathcal{U}%
\times\mathbb{P}M_{2}^{1,1}$ and coincides with the image of the entire chain
recurrent set of $\Phi$, i.e.,
\begin{equation}
\mathcal{M}_{c}^{1}=\left\{  (u,\mathbb{P}(x,1))\in\mathcal{U}\times
\mathbb{P}M_{2}^{1}\left\vert (u,x)\in\mathcal{R}^{\#}(\Phi)\right.  \right\}
. \label{E_hyp}%
\end{equation}

\end{theorem}

\begin{proof}
(i) Denote by $\mathcal{V}_{\ast}^{1}$ the right hand side of (\ref{Vstar2}).
For every $u\in\mathcal{U}$ the fiber $\left\{  (u,-re(u,0),r),r\in
\mathbb{R}\right\}  $, is one-dimensional and $\mathcal{V}_{\ast}^{1}$ is
closed. In fact, suppose that a sequence $(u^{k},-r_{k}e(u^{k},0),r_{k}%
),k\in\mathbb{N\,}$ in this set converges to $(u,x,r)\in\mathcal{U}\times
M_{2}\times\mathbb{R}$. Then $u^{k}\rightarrow u$ and $r_{k}\rightarrow r$
and, by continuity of $e(\cdot,0)$, it follows that $r_{k}e(u^{k}%
,0)\rightarrow re(u,0)$. This shows that $(u,x,r)=(u,-re(u,0),r)\in
\mathcal{V}_{\ast}^{1}$. According to Blumenthal and Latushkin \cite[Lemma
3.8]{BluL19} it follows that $\mathcal{V}_{c}^{1}=\mathcal{V}_{\ast}^{1}$ is a
one dimensional subbundle of $\mathcal{U}\times M_{2}\times\mathbb{R}$. This
subbundle is invariant, since by Proposition \ref{Proposition7.3}%
\[
\Phi_{t}^{1}(u,-re(u,0),r)=(\theta_{t}u,-re(u,t),r)=(\theta_{t}u,-re(\theta
_{t}u,0),r)\text{ for }t\in\mathbb{R}.
\]

(ii) The equality in (\ref{E_hyp}) follows from the definitions and Theorem
\ref{Theorem_hyperbolic1}. This also implies that the set $\mathcal{R}%
^{\#}(\Phi)$ is compact and hence the set $\{(u,-e(u,0),1)\in\mathcal{U}\times
M_{2}\times\mathbb{R}\left\vert u\in\mathcal{U}\right.  \}$ is also compact.
It follows that the right hand side of (\ref{E_hyp}) is compact. Finally.
Proposition \ref{Proposition_e2}(ii) shows that $\mathcal{M}_{c}^{1}%
\subset\mathcal{U}\times M_{2}^{1,1}$.
\end{proof}

This result implies the following consequences for chain control sets; cf.
Colonius and Santana \cite[Theorem 35]{ColS24} for the finite dimensional case.

\begin{corollary}
\label{Corollary_lift}Consider the delay system (\ref{D1}) and suppose that
$\det A_{p}\not =0$ and that the linear part $\Phi^{0}$ of the semiflow $\Phi$
is hyperbolic. Then the chain control set $E$ of system (\ref{D1}) is compact,
its lift $\mathcal{E}$ to $\mathcal{U}\times M_{2}$ coincides with the entire
chain recurrent set of the control semiflow $\Phi$, i.e., $\mathcal{E}%
=\mathcal{R}^{\#}(\Phi)$, and for every $u\in\mathcal{U}$ there is a unique
element $x\in E$ with $\psi(t,x,u)\in E$ for all $t\in\mathbb{R}$.
Furthermore, the image of $E$ in $\mathbb{P}M_{2}^{1}$ satisfies%
\[
\{\mathbb{P}(x,1)\left\vert x\in E\right.  \}=\{\mathbb{P}(x,1)\left\vert
\exists u\in\mathcal{U}:(u,\mathbb{P}(x,1))\in\mathcal{M}_{c}^{1}%
=\mathbb{P}\mathcal{V}_{c}^{1}\right.  \}.
\]

\end{corollary}

\begin{proof}
By Theorem \ref{Theorem_hyperbolic1}, the chain recurrent set $\mathcal{R}%
(\Phi^{\#})$ is compact, invariant, and chain transitive. Hence Theorem
\ref{Theorem_equivalence}(iii) implies that it is the lift of a chain control
set, i.e., by Theorem \ref{Theorem_ccs1} it is the lift of the unique chain
control set $E$. Then the second assertion follows by Theorem
\ref{Theorem_lift}(ii).
\end{proof}

\end{document}